\documentclass[preprint,1p,10pt]{elsarticle}

\biboptions{numbers,sort&compress}

\usepackage{lineno,hyperref}
\modulolinenumbers[1]

\usepackage{array}
\newcolumntype{C}[1]{>{\centering\arraybackslash}m{#1}}

\usepackage{graphics}
\usepackage{amssymb, latexsym, amsmath, epsfig}
\usepackage{graphicx}
 \usepackage{color}
 \usepackage{epstopdf}
 \usepackage{comment}
 \usepackage{soul}
 \usepackage[normalem]{ulem}
  \usepackage{float}

\usepackage{amssymb,amsmath}
\usepackage{amsfonts,amstext,amsthm,amssymb,comment}
\usepackage{epsfig,graphics,color}

\graphicspath{{figures/}}

\usepackage{MnSymbol}

\newtheorem{theorem}{Theorem}
\newtheorem{lemma}{Lemma}
\newtheorem{corollary}{Corollary}

\numberwithin{equation}{section}

\newtheorem{remark}{Remark}

\newcommand{\bfs}[1]{{\boldsymbol #1}}

\journal{arXiv}

\begin{document}

\begin{frontmatter}

\title{SoftIGA: soft isogeometric analysis}
%

\author[qd]{Quanling Deng\corref{corr}}
\author[pb]{Pouria Behnoudfar}
\author[vc]{Victor M. Calo}

\cortext[corr]{Corresponding author. \\ E-mail addresses: Quanling.Deng@anu.edu.au; Pouria.Behnoudfar@csiro.au; Victor.Calo@curtin.edu.au}

\address[qd]{School of Computing, Australian National University, Canberra, ACT 2601, Australia. }

\address[pb]{Mineral Resources, Commonwealth Scientific and Industrial Research Organisation (CSIRO), Kensington, Perth, WA 6152, Australia.}

\address[vc]{School of Electrical Engineering, Computing and Mathematical Sciences, Curtin University, Perth, WA 6102, Australia. }

\begin{abstract}
We extend the softFEM idea to isogeometric analysis (IGA) to reduce the stiffness (consequently, the condition numbers) of the IGA discretized problem. We refer to the resulting approximation technique as softIGA. We obtain the resulting discretization by first removing the IGA spectral outliers to reduce the system's stiffness. We then add high-order derivative-jump penalization terms (with negative penalty parameters) to the standard IGA bilinear forms. 
The penalty parameter seeks to minimize spectral/dispersion errors while maintaining the coercivity of the bilinear form. 
We establish dispersion errors for both outlier-free IGA (OF-IGA) and softIGA elements.
We also derive analytical eigenpairs for the resulting matrix eigenvalue problems and show that the stiffness and condition numbers of the IGA systems significantly improve (reduce). We prove a superconvergent result of order $h^{2p+2}$ for eigenvalues where $h$ characterizes the mesh size and $p$ specifies the order of the B-spline basis functions. 
To illustrate the main idea and derive the analytical results, we focus on uniform meshes in 1D and tensor-product meshes in multiple dimensions.
For the eigenfunctions, softIGA delivers the same optimal convergence rates as the standard IGA approximation.  Various numerical examples 
demonstrate the advantages of softIGA over IGA.

  \textbf{Mathematics Subjects Classification}: 65N15, 65N30, 65N35, 35J05 \end{abstract}
%


\end{frontmatter}


\section{Introduction} \label{sec:intr} 

Isogeometric analysis (IGA), introduced in 2005~\cite{ hughes2005isogeometric, cottrell2009isogeometric}, is a widely-used analysis tool in modeling and simulation due to its integration of the classical finite element analysis with computer-aided-design technologies.  There is vast literature; see an overview paper~\cite{ nguyen2015isogeometric} and the references therein.  In particular, a rich literature on IGA demonstrates that IGA outperforms classical finite elements in various scenarios, especially on the spectral approximation of the second-order elliptic operators.  The elliptic eigenvalue problem arises in many applications in science and engineering.  For example, when simulating the eigenmodes of structural vibrations, the work~\cite{ cottrell2006isogeometric} showed that the IGA approximated eigenenergy errors were significantly smaller compared with finite element ones.  In~\cite{ hughes2014finite}, the authors explored the further advantages of IGA on spectral approximations.

There are mainly two lines of work to further improve the IGA spectral approximation of elliptic operators: spectral error reduction in the lower-frequency region and outlier elimination.  The recent work~\cite{ calo2019dispersion, puzyrev2017dispersion} introduced optimally-blended quadrature rules that combine the Gauss-Legendre and Gauss-Lobatto rules.  In the lower-frequency region, the authors demonstrate superconvergence with two extra orders on eigenvalue errors by invoking the unified dispersion and spectral analysis in~\cite{ hughes2008duality}.  The work~\cite{ deng2018ddm} further generalized the optimally-blended rules to arbitrary $p$-th order IGA with maximal continuity.  Along this line, the work~\cite{ deng2018dispersion} studied the methods' computational efficiency, and the work~\cite{ bartovn2018generalization, calo2017quadrature, puzyrev2018spectral, deng2019optimal, deng2018isogeometric} studied their applications.

On the other hand, the outliers in IGA spectral approximations were first observed in~\cite{ cottrell2006isogeometric} in 2006.  In general, the spectral errors in the higher-frequency region are much larger than those in the lower-frequency region.  For higher-order IGA elements, there is a thin layer in the highest-frequency region where the spectral errors are significantly larger than their neighboring ones.  These approximated eigenvalues are referred to as ``outliers."  The work~\cite{ cottrell2006isogeometric, hughes2008duality} tried to remove these outliers by using a nonlinear parametrization of the geometry and meshes.  However, the advantage of using nonlinear parametrization to remove outliers is limited.

The question of how to efficiently eliminate these outliers remained open until recently.  The work~\cite{ hiemstra2021removal} proposed reconstructing the approximation space by imposing extra conditions on the boundary.  The authors studied the method numerically for the second and fourth-order problems with various boundary conditions in one, two, and three dimensions.  The work~\cite{ manni2022application} focuses on the approximation properties of the method and established optimal error estimates for both eigenvalue and eigenfunctions.  In these works, the extra boundary conditions are imposed strongly in the approximation space.  Alternatively, these conditions can be imposed in a weak sense~\cite{ Bazilevs:2007, Bazilevs:2008}.  The work~\cite{ deng2021boundary} removed the outliers by adding to the bilinear form a boundary penalization term.  With optimally-blended quadratures, the work~\cite{ deng2021outlier} removed the outliers in the higher-frequency region as well as reduced the spectral errors (with superconvergent errors) in the lower-frequency region.  All these methods eliminated the outliers, where the key idea was to impose higher-order consistency conditions on the boundaries in the IGA approximations.
These methods are referred to outlier-free IGA (OF-IGA).

As a consequence of the outlier elimination, the stiffness and condition numbers of the discretized system improve~\cite{ deng2021outlier}.  This paper proposes further reducing the stiffness and condition numbers by subtracting a high-order derivative-jump penalization term at the internal element interfaces.  This method extends the main idea of softFEM developed in~\cite{ deng2021softfem} to the IGA setting.  We refer to the corresponding analysis as softIGA.  For $p$-th order IGA elements with maximal continuity, the basis functions are $C^{p-1}$-continuous.  The jumps appear when taking the basis functions' $p$-th order partial derivatives.  We thus penalize this $p$-th order derivative-jump and subtract from the outlier-free IGA bilinear form~\cite{ hiemstra2021removal, manni2022application, deng2021outlier} an inner product of the derivative-jumps of the basis functions in both trial and test spaces. 
 The jump terms are scaled by a softness parameter $\eta$. We show that the softIGA bilinear form is coercive for $\eta \in [0, \eta_{\max})$ where $\eta_{\max}$ depends on $p$ and mesh configuration. 
 With the coercivity and boundedness of the jump terms in mind, one expects optimal eigenvalue and eigenfunction error convergence. 
In particular, in a 1D setting with uniform elements for $p=2,3,4$, we derive the analytical eigenvalues and eigenvectors for the resulting matrix eigenvalue problems. 
The eigenvalue errors are optimal and we give exact constants that bound the errors for each eigenvalue.
Lastly, we derive superconvergent eigenvalue errors of orders $2p+2$ when using a particular choice of the softness parameter.
We observe a superconvergence of order $2p+4$ when we also add the penalized jump bilinear term to the mass bilinear form.

SoftIGA requires the evaluation and implementation of high-order derivative jumps in existing IGA codes.
This implementation is straightforward by using the Cox-de Boor recursive formula with repeated nodes.
With this simple code extension, we observe stiffness and condition number reduction on the resulting matrices.
This work is our first study on softIGA towards reducing the spectral error and stiffness for general $C^k$ $p$-th order IGA elements. 
For $C^1$ quadratic IGA elements, there is no outlier in the spectra and softIGA reduces the condition number by about 50\%, which is similar than the softFEM case.
For $C^{p-1}$ $p$-th ($p>2$) order OF-IGA elements, the stiffness and error reduction are decreasing with respect to $p$ as the element continuity increases.
They are not as much as the case in the softFEM setting as there are no optical branches in OF-IGA spectra. 
For general $C^k, k<p-1,$ $p$-th order IGA elements, there are optical branches in the spectra and we expect larger reductions. 
For condition number estimates, we derive analytical eigenpairs for OF-IGA and softIGA with $p=2,3,4$ and uniform 1D elements. 
From the exact eigenpairs, we observe that 
softIGA eigenvalues errors are smaller in lower-frequency region and larger in the high-frequency region.
The eigenvectors are the same (independent of $\eta$), thus the eigenfunctions are of the same errors in both $H^1$ and $L^2$ norm.
As a by-product, we establish dispersion errors for softIGA elements where OF-IGA is a special case when $\eta=0$ in softIGA.

The rest of this paper is organized as follows.  Section~\ref{sec:iga0} first presents the eigenvalue problem followed by its discretization using the standard isogeometric analysis.  We then describe the reconstructed B-spline basis function and present the recently-developed outlier-free IGA (OF-IGA).
Section~\ref{sec:softiga} introduces softIGA followed by its study on the choice of the softness parameter and establishes the coercivity of the new bilinear form.  In Section~\ref{sec:sr}, we focus on the Laplacian eigenvalue problem in 1D and derive analytical eigenpairs for the softIGA resulting matrix systems.
We perform the dispersion error analysis in Section~\ref{sec:de}.  Section~\ref{sec:num} collects numerical results demonstrating the proposed method's performance.  Concluding remarks are presented in Section~\ref{sec:conclusion}.

\section{The standard and outlier-free isogeometric analysis} \label{sec:iga0}

In this section, we discuss the eigenvalue problem and its variational formulation. We then present the standard IGA discretization followed by the description of the outlier-free (OF) IGA~\cite{ hiemstra2021removal, manni2022application}. The key idea for outlier elimination is to reconstruct the B-spline space such that the functions in the test and trial spaces satisfy extra boundary conditions. We will adopt the reconstructed outlier-free approximation space and introduce the soft isogeometric analysis in the next section.

\subsection{Problem statement} \label{sec:ps}

We begin our introduction of softIGA with the Laplacian eigenvalue problem posed on the domain $\Omega = [0,1]^d \subset \mathbb{R}^d, d=1,2,3$ with Lipschitz boundary $\partial \Omega$: Find the eigenpairs $(\lambda, u)\in \mathbb{R}^+\times H^1_0(\Omega)$ with $\|u\|_\Omega=1$ such that
\begin{equation} \label{eq:pde}
  \begin{aligned}
    - \Delta u & =  \lambda u \quad &&  \text{in} \quad \Omega, \\
    u & = 0 \quad && \text{on} \quad  \partial \Omega,
  \end{aligned}
\end{equation}
where $\Delta = \nabla^2$ is the Laplacian. 
We adopt the standard notation for the Hilbert and Sobolev spaces.  In particular, we denote, for a measurable subset $S\subseteq \Omega$, by $(\cdot,\cdot)_S$ and $\| \cdot \|_S$ the $L^2$-inner product and its norm, respectively.  For an integer $m\ge1$, let $\| \cdot \|_{H^m(S)}$ and $| \cdot |_{H^m(S)}$ denote the $H^m$-norm and $H^m$-seminorm, respectively.  Let $H^1_0(\Omega)$ be the Sobolev space with functions in $H^1(\Omega)$ that are vanishing at the boundary.

The variational formulation of~\eqref{eq:pde} at the continuous level is to find the eigenvalue $\lambda \in \mathbb{R}^{+}$ and the associated eigenfunction $u \in H^1_0(\Omega)$ with $\|u\|_\Omega=1$ such that
\begin{equation} \label{eq:vf}
  a(w, u) =  \lambda b(w, u), \quad \forall \ w \in H^1_0(\Omega), 
\end{equation}
where the bilinear forms are
\begin{equation}
  a(v,w) := (\nabla v, \nabla w)_\Omega,
  \qquad
  b(v,w) := (v,w)_\Omega.
\end{equation}

The eigenvalue problem~\eqref{eq:vf} is equivalent to the original problem~\eqref{eq:pde}.  They have a countable set of positive eigenvalues (see, for example, \cite[Sec. 9.8]{Brezis:11})
\begin{equation*}
  0 < \lambda_1 < \lambda_2 \leq \lambda_3 \leq \cdots
\end{equation*}
and an associated set of orthonormal eigenfunctions $\{ u_j\}_{j=1}^\infty$, meaning, $ (u_j, u_k) = \delta_{jk}, $ where $\delta_{jk} =1$ is the Kronecker delta.  Since there holds $ a(u_j, u_k) = \lambda_j b(u_j, u_k) = \lambda_j \delta_{jk}, $ the eigenfunctions are also orthogonal in the energy inner product.  Throughout the paper, we always sort the eigenvalues, paired with their corresponding eigenfunctions, in ascending order and counted with their order of algebraic multiplicity.

\subsection{Isogeometric analysis (IGA)} \label{sec:iga}

Standard IGA adopts the Galerkin finite element analysis framework at the discrete level. We first discretize the rectangular domain $\Omega = [0,1]^d, d=1,2,3$ with a tensor-product mesh.  Let $E$ and $\mathcal{T}_h$ denote a general element and its collection, respectively, such that $\overline \Omega = \cup_{E\in \mathcal{T}_h} E$.  Also, let $h_E = \text{diameter}(E)$ and $h = \max_{E \in \mathcal{T}_h} h_E$.  We now use the B-splines as basis functions for simplicity, which we construct using the Cox-de Boor recursive formula in 1D (see, for example,~\cite{ de1978practical, piegl2012nurbs})
\begin{equation} \label{eq:Bspline}
  \begin{aligned}
    \phi^j_0(x) & = 
    \begin{cases}
      1, \quad \text{if} \ x_j \le x < x_{j+1}, \\
      0, \quad \text{otherwise}, \\
    \end{cases} \\ 
    \phi^j_p(x) & = \frac{x - x_j}{x_{j+p} - x_j} \phi^j_{p-1}(x) + \frac{x_{j+p+1} - x}{x_{j+p+1} - x_{j+1}} \phi^{j+1}_{p-1}(x),
  \end{aligned}
\end{equation}
where $\phi^j_p(x)$ is the $j$-th B-spline basis function of degree $p$.  Herein, the knot vector is $X = \{0=x_0,\cdots, x_0, x_1, x_2, \cdots, x_N, \cdots x_N = 1\}$ which has a non-decreasing sequence of real numbers $x_j$.
For $C^{p-1}$ $p$-th order B-splines, the left and right boundary nodes are repeated $p$ times.  The multi-dimensional basis functions construction uses tensor products of these one-dimensional functions.  
We refer to \cite{ de1978practical}
for details on this construction.

Taking the boundary condition into consideration, the usual IGA approximation space (see also \cite{hughes2005isogeometric}) associated with the knot vector $X$ is $V^h_p = \text{span} \{\phi^j_p(x)\}_{j \in I_h} \subset H^1_0(\Omega)$.  Herein, $I_h$ is an index set such that the associated basis functions vanish at the boundary. Lastly, throughout the paper, we focus on the IGA approximation spaces with B-splines of maximal continuity.

The isogeometric analysis (IGA) of~\eqref{eq:pde} or equivalently~\eqref{eq:vf} seeks an eigenvalue $\lambda^h \in \mathbb{R}^+$ and its associated eigenfunction $u^h \in V^h_p$ with $\| u^h \|_\Omega = 1$ such that
\begin{equation} \label{eq:vfh}
  a(w^h, u^h) =  \lambda^h b(w^h, u^h), \quad \forall \ w^h \in V^h_p.
\end{equation} 
At the algebraic level, an eigenfunction is a linear combination of the B-spline basis functions.  Substituting all the B-spline basis functions for $w^h$ in~\eqref{eq:vfh} leads to the generalized matrix eigenvalue problem (GMEVP)
 \begin{equation} \label{eq:mevp}
   KU = \lambda^h MU,
\end{equation}
where $K_{kl} = a(\phi_p^l, \phi_p^k), M_{kl} = b(\phi_p^l, \phi_p^k),$ and $U$ is the eigenvector representing the coefficients of the B-spline basis functions.  Once the matrix eigenvalue problem~\eqref{eq:mevp} is solved, we sort the eigenpairs $(\lambda^h_j, U_j), j\in I_h$ in ascending order such that the high-frequency eigenfunctions have larger indices.

\subsection{Outlier-free isogeometric analysis (OF-IGA)}
\label{eq:ofiga}

The outliers, first observed in~\cite{ cottrell2006isogeometric} for high-order elements, are the approximate eigenpairs in the highest-frequency region where the eigenvalues are much larger than the exact values. The eigenvalue errors of the outliers are substantially larger than those in the lower-frequency region. Recently, these outliers have been effectively removed by including extra accuracy at the boundaries~\cite{ hiemstra2021removal, manni2022application, deng2021boundary}. In particular,~\cite{ hiemstra2021removal, manni2022application} impose the extra conditions strongly through the re-construction of the B-spline basis functions near the boundary, which leads to a smaller approximation space while~\cite{ deng2021boundary} imposes extra condition weakly without space re-construction. The work~\cite{ hiemstra2021removal} studied the method numerically for both the second- and fourth-order problems with various boundary conditions in one, two, and three dimensions. The work~\cite{ manni2022application} focuses on the approximation properties of the method and establishes optimal error estimates for both eigenvalue and eigenfunctions. For the development of softIGA, we adopt the strong basis-function re-construction in~\cite{ hiemstra2021removal, manni2022application}.

\begin{figure}[h!]
\centering
\includegraphics[height=7cm]{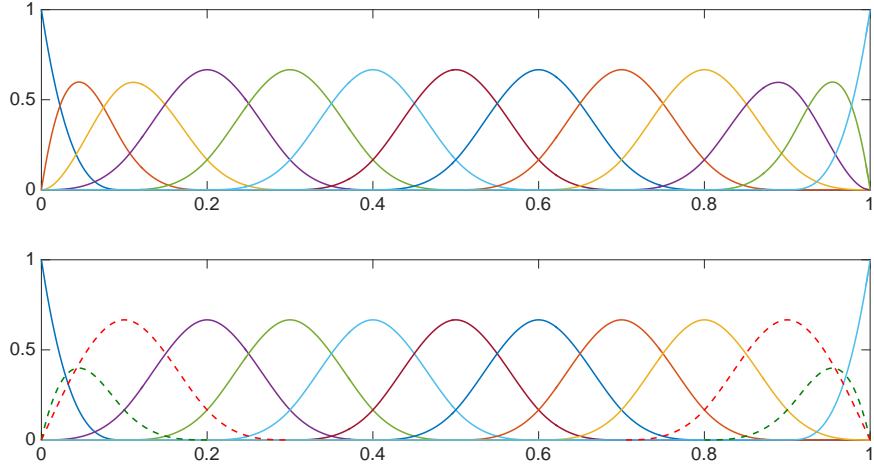} 
\caption{Original (top) and new (bottom) $C^2$-cubic B-splines with $N=10$ uniform elements. The $C^2$-cubic OF-IGA space with 10 uniform elements is spanned by the 3rd-to-11th basis functions shown at the bottom (the dashed lines indicate the new basis functions that are different from standard IGA).}
\label{fig:c2p3n10}
\end{figure}

\begin{figure}[h!]
\centering
\includegraphics[height=8cm]{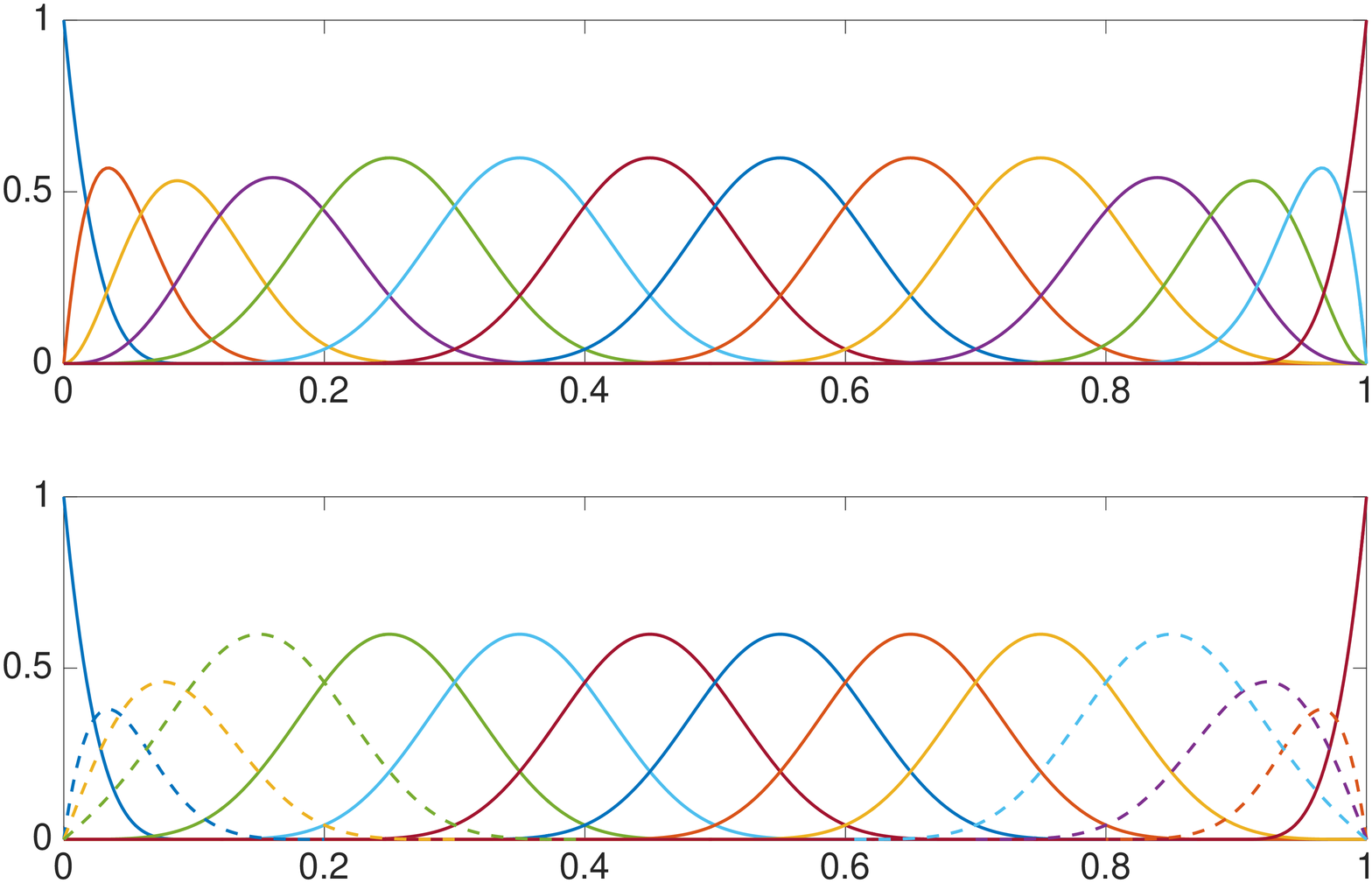} 
\caption{Original (top) and new (bottom) $C^3$-quartic B-splines with $N=10$ uniform elements. The $C^3$-quartic OF-IGA space with 10 uniform elements is spanned by the 3rd-to-12th basis functions shown at the bottom (the dashed lines indicate the new basis functions that are different from standard IGA).}
\label{fig:c3p4n10}
\end{figure}

\begin{figure}[h!]
\centering
\includegraphics[height=8cm]{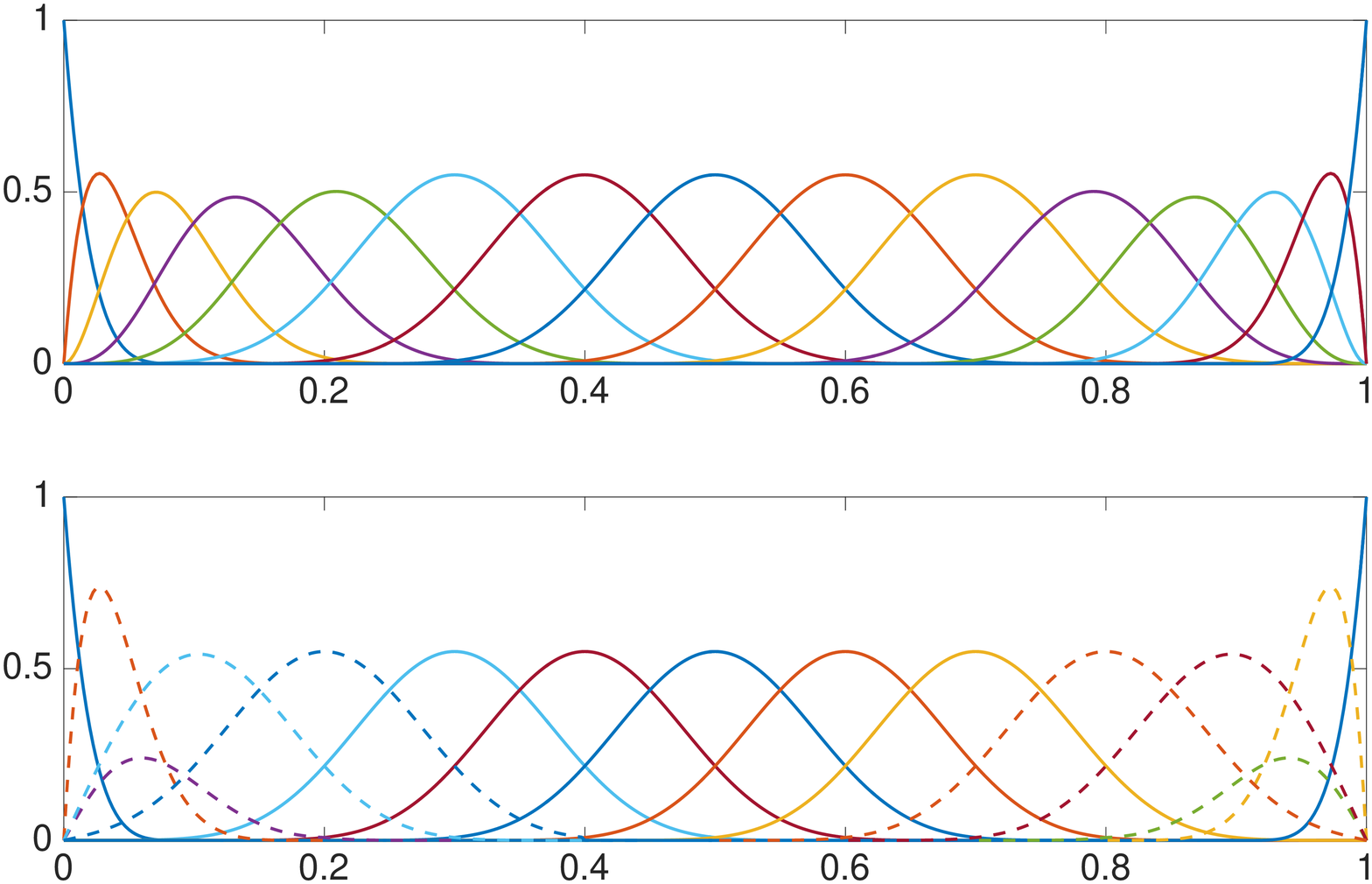} 
\caption{Original (top) and new (bottom) $C^4$-quintic B-splines with $N=10$ uniform elements. The $C^4$-quintic OF-IGA space with 10 uniform elements is spanned by the 4th-to-12th basis functions shown at the bottom (the dashed lines indicate the new basis functions that are different from standard IGA).}
\label{fig:c4p5n10}
\end{figure}

The key idea of OF-IGA in strong form is to construct an outlier-free spline space $\tilde{V}^h_p$, which is a subspace of the standard IGA spline space $V^h_p$. Thus, $\tilde{V}^h_p \subset V^h_p \subset H^1_0(\Omega)$. Specifically, the OF-IGA space $\tilde{V}^h_p$ for the Laplacian eigenvalue problem~\eqref{eq:pde} in 1D is defined as follows~\cite{ hiemstra2021removal}:
\begin{equation}
  \tilde{V}^h_p = \{ w \in C^{p-1}(\Omega): w|_E \in \mathbb{P}_p(E), \mathcal{L}^\alpha (w)|_{\partial \Omega} = 0, \alpha=0,1,\cdots, \alpha_p, \forall E \in \mathcal{T}_h \},
\end{equation}
where $\mathbb{P}_p$ is a space of polynomial of order $p$. The space $ \tilde{V}^h_p$ for multiple dimensions is defined using the tensor-product structural.  Herein, $\mathcal{L} = \Delta$ in 1D and $\alpha_p$ is defined as
\begin{equation}
  \alpha_p = \lfloor \frac{p-1}{2} \rfloor =
  \begin{cases}
    \frac{p-1}{2}, & p \quad \text{is odd}, \\
    \frac{p-2}{2}, & p \quad \text{is even}. \\
  \end{cases}
\end{equation}
For $p=1,2$, the condition $\mathcal{L}^\alpha (w)|_{\partial \Omega} = 0$ reduces to the homogeneous Dirichlet boundary condition in~\eqref{eq:pde}. Thus, $\tilde{V}^h_p \equiv V^h_p$. This extra penalization corresponds to the case of linear and $C^1$-quadratic IGA elements where there are no outliers in the approximate spectra.  For OF-IGA, we focus on $p>2$.  For $p=3,4,5$, the new B-spline basis functions (the internal ones remain the same) in 1D are shown and compared with the original B-splines in Figures~\ref{fig:c2p3n10},~\ref{fig:c3p4n10}, and~\ref{fig:c4p5n10}, respectively.

For $p=3$ in 1D, the space $\tilde{V}^h_p$ with $N=10$ uniform elements is spanned by the set of basis functions consisting of all but the first and last two new B-splines shown in Figure~\ref{fig:c2p3n10}.  Similarly, for $p=4$ in 1D, the space $\tilde{V}^h_p$ is spanned by the set of basis functions consisting of all but the first and last two new B-splines shown in Figure~\ref{fig:c3p4n10}.  In general, after removing the basis functions associated with the boundary conditions, the space $\tilde{V}^h_p$ has $N-1$ (new) B-spline basis functions for an odd order $p\ge1$ while $N$ B-spline basis functions for an even order $p\ge2$.  We refer to~\cite{ hiemstra2021removal, manni2022application, floater2019optimal, sande2019sharp} for details of the construction of new B-spline basis functions and the corresponding approximation space.  The standard IGA space can be written as $V^h_p = \{ w \in C^{p-1}(\Omega): w|_E \in \mathbb{P}_p, w|_{\partial \Omega} = 0, \forall E \in \mathcal{T}_h \}$ so that the relation $\tilde{V}^h_p \subset V^h_p$ is clear.  With the outlier-free approximation space in hand, the OF-IGA is to find $\tilde{\lambda}^h \in \mathbb{R}^+$ and $\tilde{u}^h \in V^h_p$ with $\| \tilde{u}^h \|_\Omega = 1$ such that
\begin{equation} \label{eq:vfhh}
a(w^h, \tilde{u}^h) =  \tilde{\lambda}^h b(w^h, \tilde{u}^h), \quad \forall \ w^h \in \tilde{V}^h_p,
\end{equation} 
which leads to a GMEVP (in a similar fashion as described in section~\ref{sec:iga})
\begin{equation} \label{eq:mevpp}
\tilde{K} \tilde{U} = \tilde{\lambda}^h \tilde{M} \tilde{U}.
\end{equation}
Lastly, we remark that for $C^0$-linear and $C^1$-quadratic elements, OF-IGA reduces to the standard IGA.

\section{Soft isogeometric analysis (SoftIGA)} \label{sec:softiga}

The softFEM, introduced in~\cite{ deng2021softfem}, removes the stopping bands appearing in the FEM spectra and, more importantly, reduces the stiffness and condition numbers of the discrete systems.  To extend the idea to IGA with B-splines, we reduce the stiffness by subtracting a penalized term on the higher-order derivative jumps due to the higher-order continuities.  For $C^{p-1}, p$-th-order elements, the $j$-th $j=0,1,\cdots, p-1$, derivatives are continuous.  Thus, there are no jumps on these derivatives.  The $p$-th order derivatives are constants on each element.  They are discontinuous at the element interfaces.  We thus impose a penalty on $p$-th order derivative jumps.  This penalization is the essential novelty of the proposed penalization technique compared to the ones in the softFEM and discontinuous Galerkin (DG) methods.

\subsection{SoftIGA}

We introduce softIGA by building on the OF-IGA approximation space $\tilde{V}^h_p$. For a tensor-product mesh $\mathcal{T}_h$, let $F$ denote a face while $\mathcal{F}$ represents the set of interior faces of the mesh. Each $F\in \mathcal{F}$ specifies an interface of two elements. Similarly, let $F_b$ denote a face at the boundary while $\mathcal{F}_b$ denote the set of boundary faces of the mesh. We introduce the $p$-th order derivative-jump at an interior interface of two neighboring elements $E_1$ and $E_2$:
$$
\lsem \hat\nabla^p v \cdot \bfs{n} \rsem = \hat\nabla^p v |_{E_1} \cdot \bfs{n}_1 + \hat\nabla^p v |_{E_2} \cdot \bfs{n}_2, \quad \forall v \in \tilde{V}^h_p(\mathcal{T}_h),
$$
where $\hat\nabla^p = (\partial^p_{x_1}, \cdots, \partial^p_{x_d})^T$, i.e., a vector of $p$-th order partial derivatives in each dimension.  $\bfs{n}_1$ and $\bfs{n}_2$ are the outward unit normals of elements $E_1$ and $E_2$, respectively.  For a boundary interface $F_b$ associated with element $E$, we define the jump as
$$
\lsem \hat\nabla^p v \cdot \bfs{n}_{F_b} \rsem = \hat\nabla^p v |_{E} \cdot \bfs{n}_{F_b},
$$
where $\bfs{n}_{F_b}$ is the outward unit normal of element $E$ at the boundary interface $F_b$. This simple construction is possible due to the isogeometric framework, which relies on isoparametric geometric descriptions with highly continuous geometric maps. 
Herein, we focus on the Cartesian setting with tensor-product meshes for simplicity to describe the main idea.

The softIGA is to find $\hat{\lambda}^h \in \mathbb{R}^+$ and $\hat{u}^h \in V^h_p$ with $\| \hat{u}^h \|_\Omega = 1$ such that
\begin{equation} \label{eq:vfhhh}
  a(w^h, \hat u^h) - \eta s(w^h, \hat u^h) =  \hat\lambda^h b(w^h, \hat u^h), \quad \forall \ w^h \in \tilde{V}^h_p,
\end{equation} 
where for $w, v \in \tilde V^h_p$ we define the softness bilinear form as
\begin{equation} \label{eq:bfs}
  s(w, v)  = 
  \begin{cases}
\vspace{.5cm}
    \sum_{F \in \mathcal{F}} h^{2p-1} (\lsem \hat\nabla^p w \cdot \bfs{n} \rsem, \lsem\hat\nabla^p v \cdot \bfs{n} \rsem ),  & p \quad \text{is odd}, \\
    \sum_{F \in \mathcal{F}} h^{2p-1} (\lsem \hat\nabla^p w \cdot \bfs{n} \rsem, \lsem\hat\nabla^p v \cdot \bfs{n} \rsem ) \\
    \qquad + 2\sum_{F_b \in \mathcal{F}_b} h^{2p-1} (\lsem \hat\nabla^p w \cdot \bfs{n} \rsem, \lsem\hat\nabla^p v \cdot \bfs{n} \rsem ) ,  & p \quad \text{is even}. \\
  \end{cases}
\end{equation}
Herein, $\eta\ge0$ is the softness parameter. When $\eta=0$, this reduces to OF-IGA. We set $\eta \in (0, \eta_{\max}]$ where $\eta_{\max}$ is to be determined such that $\hat{a}(w^h, \hat u^h) = a(w^h, \hat u^h) - \eta s(w^h, \hat u^h) $ is coercive.  The term $\eta s(w^h, \hat u^h)$ determines how much stiffness the softIGA system will reduce from the OF-IGA discretized system. Similarly, the softIGA formulation~\eqref{eq:vfhhh} leads to a GMEVP
\begin{equation} \label{eq:mevppp}
  \hat K \hat{U} = \hat{\lambda}^h \tilde{M} \hat{U}, \qquad \hat K := \tilde{K} -\eta S.
\end{equation}

\begin{remark}[Softness bilinear form $p$-dependence]
  The softness bilinear form $s(\cdot, \cdot)$ is defined in~\eqref{eq:bfs} differently for odd- or even-order elements.  The difference lies in the penalty on the boundary interfaces, which is consistent with the OF-IGA setting.  More importantly, it is consistent with the dispersion analysis with uniform elements in $\Omega=[0,1]$.  For odd-order elements, a basis function corresponds to a mesh node at the interface $x_j$ and its Bloch wave assumption takes the form $\sin( jh \omega_k)$, while for even-order elements, a basis function corresponds to a mesh middle point at $x_{j-1/2}$ and its Bloch wave assumption takes the form $\sin( (j-1/2)h \omega_k)$(cf.,~\cite{ hughes2014finite}).  Consequently, this setting leads to the desired Toeplitz-plus-Hankel matrices~\cite{ strang2014functions} which analytical eigenpairs can be derived~\cite{ deng2021analytical}.  We show coercivity for both cases in the next section.
\end{remark}

\subsection{Coercivity of softIGA bilinear form and error estimates}

Before we show coercivity, we first present the following well-known inverse inequality (see~\cite[Lemma 4]{sande2022ritz} or~\cite{ goetgheluck1990markov})

\begin{lemma}[Inverse inequality]\label{lem:inv}
Let $v \in \mathbb{P}_p([0, h])$. There holds
\begin{equation} \label{eq:inveq}
\| v' \|_{L^2([0,h])} \le C_{p,1} h^{-1}  \| v \|_{L^2([0,h])}, \quad C_{p,1} := \sqrt{ \frac{p(p+1)(p+2)(p+3)}{2} },
\end{equation}
where $v'$ denotes the first-derivative of v and $C_{p,1}$ is independent of $h$. 
\end{lemma}

A shaper constant $C_{p,1}$ is derived in~\cite{ ozisik2010constants, goetgheluck1990markov}.  However, these sharper constants are established only for lower-order elements and are not in a general form of $p$.  Herein, for the completeness of $p$, we apply~\eqref{eq:inveq}.  Using this inequality recursively, one obtains the following.

\begin{corollary}[Higher-order inverse inequality]\label{coro:inv}
  Let $v \in \mathbb{P}_p([0, h])$ and $k\le p$. There holds
  \begin{equation}
    \| v^{(p)} \|_{L^2([0,h])} \le C_{p,2} h^{k-p}  \| v^{(k)} \|_{L^2([0,h])}, 
  \end{equation}
  where
  \begin{equation}
    C_{p,2} := \sqrt{ \frac{(p-k)! \cdot (p-k+1)! \cdot (p-k+2)! \cdot (p-k+3)! }{3 \cdot 2^{p-k+2} }}
  \end{equation}
  and $C_{p,2}$ is independent of $h$. 
\end{corollary}

We now establish the following result.

\begin{lemma}[Discrete trace inequality, cuboid] \label{lem:tracegrad}
  Let $\tau := [0, h_1] \times \ldots \times [0, h_d] \subset \mathbb{R}^d$ with $h_j > 0$ for all $j\in \{1,\ldots,d\}$ be a cuboid with boundary $\partial \tau$ and outward normal $\bfs{n}_\tau$.  Let $h_\tau^0:=\min_{i\in\{1,\ldots,d\}} h_i$ be the length of the smallest edge of $\tau$. The following holds:
  \begin{equation} \label{eq:trace_cuboid}
    \| \hat\nabla^p v \cdot \bfs{n}_\tau  \|_{\partial \tau} \le \sqrt{2} (h_\tau^0)^{-1/2} \|  \hat\nabla^p v   \|_{\tau}, \quad \forall v \in \mathbb{P}_p(\tau).
  \end{equation}
\end{lemma}

\begin{proof}
  We denote $\partial \tau = \cup_{j=1}^d \mathcal{F}_{x_j}$ where $\mathcal{F}_{x_j}$ contains two faces located at $x_j =0, h_j$.  We note that the $p$-th order partial derivative $\partial^p_{x_j} v$ is constant with respect to $x_j$.  Thus, we have
\begin{equation*} 
\begin{aligned}
\| \hat\nabla^p v \cdot \bfs{n}\|_{\mathcal{F}_{x_j}}^2 & = \int_{\mathcal{F}_{x_j}} (\partial^p_{x_j} v |_{x_j=0})^2 \ ds +  \int_{\mathcal{F}_{x_j}} (\partial^p_{x_j} v |_{x_j=h_j})^2 \ ds   \\
& = 2 h_j^{-1} \int_\tau (\partial^p_{x_j} v )^2 \ d\bfs{x}  \\
& = 2h_j^{-1}   \| \partial^p_{x_j} v  \|_{ \tau}^2,
\end{aligned}
\end{equation*}
where $ds$ specifies a surface differential.  Summing the above equalities and then taking the square root yield the desired result. 
\end{proof}

The equality can be attained when $h_j=h_k$ for all $j, k \in \{1,\ldots,d\}$ or when the right-hand side is written in a summation over each dimension.  Combining Corollary~\ref{coro:inv} using $k=1$ with Lemma~\ref{lem:tracegrad}, we obtain the following result.

\begin{corollary}[Trace-inverse inequality]\label{coro:t}
With the setting in Lemma~\ref{lem:tracegrad}. There holds
\begin{equation} 
  \| \hat\nabla^p v \cdot \bfs{n}_\tau  \|_{\partial \tau} \le C_{p,3} (h_\tau^0)^{1/2-p} \| \nabla v  \|_{\tau}, \quad \forall v \in \mathbb{P}_p(\tau),
\end{equation}
where
\begin{equation} \label{eq:c3}
C_{p,3} := \sqrt{ \frac{(p-1)! \cdot p! \cdot (p+1)! \cdot (p+2)! }{3 \cdot 2^p }}
\end{equation}
and $C_{p,3}$ is independent of $h$. 
\end{corollary}

Figure \ref{fig:cp3} shows the growth of $C_{p,3}$ with respect to $p$ in a semi-logarithmic scale. 
We now establish the coercivity of the new bilinear form. 

\begin{figure}[h!]
\centering
\includegraphics[height=6cm]{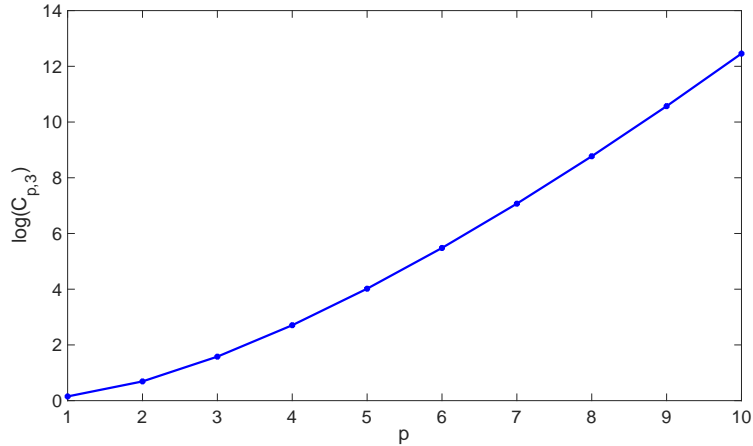} 
\caption{Growth of $C_{p,3}$ with respect to $p$ for $p = 1,2,\cdots, 10$.}
\label{fig:cp3}
\end{figure}

\begin{theorem}[Coercivity]\label{thm:coe}
  For softIGA bilinear form defined in~\eqref{eq:vfhhh}, let $\eta_{\max} = \frac{1}{2C^2_{p,3}}$ where $C_{p,3}$ is defined in~\eqref{eq:c3} and let the softness parameter $\eta\in [0,\eta_{\max})$. There holds:
\begin{equation}
  \hat a(w, w) \ge \beta | w |^2_{H^1(\Omega)}, \qquad \forall w \in \tilde{V}^h_p(\mathcal{T}_h), 
\end{equation} 
with $\beta := 1-\frac{\eta}{\eta_{\max}} >0$.
\end{theorem}

\begin{proof}
  Following~\cite[\S 5.2]{ deng2021softfem}, we denote by $\mathcal{T}_F$ be the set collecting the two mesh elements sharing $F\in\mathcal{F}$.  By softness bilinear form definition~\eqref{eq:bfs}, we calculate first for odd $p$
\[
s(w,w) = \sum_{F \in \mathcal{F}} h^{2p-1} (\lsem \hat\nabla^p w \cdot \bfs{n} \rsem, \lsem\hat\nabla^p w \cdot \bfs{n} \rsem )
\le 2\sum_{F \in \mathcal{F} }\sum_{E\in \mathcal{T}_F} h^{2p-1}  \| \hat\nabla^p w |_E \cdot \bfs{n}_\tau\|_F^2.
\]
For even $p$, we have 
\[
s(w,w) 
\le 2\sum_{F \in \mathcal{F} }\sum_{E\in \mathcal{T}_F} h^{2p-1}  \| \hat\nabla^p w |_E \cdot \bfs{n}_\tau\|_F^2 + 2\sum_{F_b \in \mathcal{F}_b} h^{2p-1} \| \hat\nabla^p w |_{E_{F_b}} \cdot \bfs{n}_\tau\|_{F_b}^2,
\]
where $E_{F_b}$ is the boundary element containing $F_b$.  For both cases, exchanging the order of the two summations leads to
\[ 
s(w,w) \le 2 \sum_{E \in \mathcal{T}_h} h^{2p-1}  \| \hat\nabla^p w \cdot \bfs{n}_E \|_{\partial_E}^2.
\]
Applying Corollary~\ref{coro:t} yields $s(w,w) \le 2C^2_{p,3} a(w, w)$. Consequently, we have
\begin{equation} \label{eq:hat_a_a}
\hat a(w, w) = a(w, w) - \eta s(w,w) \ge (1 - 2 C^2_{p,3} \eta) a(w,w) = (1 - 2 C^2_{p,3} \eta) | w |^2_{H^1(\Omega)},
\end{equation}
which leads to the desired result.
\end{proof}

\begin{remark}[Sharpness of $\eta_{\max}$]
  The constants $C_{p,1}$ and $C_{p,3}$ are sharp for $v \in \mathbb{P}_p([0, h])$ and $v \in \mathbb{P}_p(\tau)$, respectively.  These constants may not be sharp for the splines $v \in \tilde{V}^h_p(\mathcal{T}_h)$.  Numerical experiments show that $\eta_{\max} = \frac{1}{2C^2_{2,3}} = \frac{1}{48}$ is sharp to guarantee the coercivity for $C^1$ quadratic uniform softIGA elements.  For $C^2$-cubic and $C^3$-quartic uniform softIGA elements, the sharp values (on uniform mesh) are $\eta_{\max} = \frac{1}{480} > \frac{1}{2C^2_{3,3}} = \frac{1}{2,880}$ and $\eta_{\max} = \frac{17}{80,640} > \frac{1}{2C^2_{4,3}} = \frac{1}{259,200}$, respectively.  We detail these estimates in the next section.
\end{remark}

Lastly, the penalty term admits consistency in that the softness bilinear $s(\cdot, \cdot)$ vanishes for analytic solutions of $C^p(\Omega)$.  This property guarantees the Galerkin orthogonality for softIGA of the corresponding source problem $-\Delta u = f$.  Thus, consistency and coercivity allow us to expect optimal approximation properties for eigenvalues and smooth eigenfunctions:
\begin{equation} \label{eq:ee}
  \big| \hat \lambda_j^h - \lambda_j \big|  \le C_{\eta,j}h^{2p}, \qquad | u_j - \hat u_j^h |_{H^1(\Omega)} \le C_{\eta,j} h^{p}, 
\end{equation}
where $C_{\eta,j}$ is a positive constant independent of the mesh-size $h$ and depending on $\eta$ and the $j$-th eigenpair.  Following \cite[\S~6]{strang1973analysis} and~\cite{ deng2021softfem}, one can obtain the Pythagorean identity for each eigenpair $(\hat{\lambda}^h_j, \hat{u}_j^h)$
\begin{equation*} 
||| u_j - \hat u_j^h |||^2 = \lambda_j \| u_j - \hat u_j^h \|^2_{\Omega} + \hat{\lambda}_j^h - \lambda_j,
\end{equation*}
where $||| \cdot ||| := \hat a(\cdot, \cdot).$ Similarly, by following~\cite{ deng2021softfem}, one can derive the eigenvalue lower and upper bounds, 
\begin{equation}\label{eq:bds}
(1-2\eta C^2_{p,3}) \tilde{\lambda}_j^h \le \hat \lambda_j^h < \tilde{\lambda}_j^h, \qquad \eta\in (0,\eta_{\max}).
\end{equation}

\subsection{Stiffness reduction} \label{sec:s}

The condition numbers of the resulting matrix eigenvalue problems characterizes the stiffness of the discrete systems~\eqref{eq:vfh},~\eqref{eq:vfhh}, and~\eqref{eq:vfhhh}.  The stiffness and mass matrices are symmetric, thus the condition numbers for the systems~\eqref{eq:vfh},~\eqref{eq:vfhh}, and~\eqref{eq:vfhhh} are
\begin{equation}
  \gamma := \frac{\lambda^h_{\max}}{\lambda^h_{\min}}, \qquad \tilde{\gamma} := \frac{\tilde{\lambda}^h_{\max}}{\tilde{\lambda}^h_{\min}}, \qquad \text{and} \qquad \hat{\gamma} := \frac{\hat{\lambda}^h_{\max}}{\hat{\lambda}^h_{\min}}, 
\end{equation}
respectively.  We focus on softIGA as OF-IGA is a special case when $\eta = 0$.

Following the definitions in~\cite{ deng2021softfem} for softFEM, we define the \textit{condition number reduction ratio} of softIGA with respect to IGA as
\begin{equation} \label{eq:srr}
  \rho^h := \frac{\gamma}{\hat\gamma} = \frac{\lambda^h_{\max} }{\hat \lambda^h_{\max} } \cdot \frac{\hat \lambda^h_{\min} }{\lambda^h_{\min}}.
\end{equation}
In general, $\rho^h$ depends on the mesh and element order. We denote $\rho = \lim_{h\to 0} \rho^h.$ Moreover, the smallest eigenvalue is approximated well with a few elements. Thus, one has $\lambda^h_{\min} \approx \hat \lambda^h_{\min}$ when using a few elements.  Consequently, the ratio of the largest eigenvalues characterizes the reduction ratio. Lastly, we define the \textit{condition number reduction percentage} as
\begin{equation} \label{eq:srp}
  \varrho^h = 100 \frac{\gamma-\hat \gamma }{\gamma}\,\% = 100(1-1/\rho^h) \, \%
\end{equation} 
and its asymptotic percentage as $\varrho = \lim_{h\to 0} \varrho^h.$

We specify these values for the $p$-th order element using a subscript of $p$. For example, we use $\varrho_p$ to denote the reduction percentage for $p$-th order element. Next, we derive some analytical and asymptotic results.

\section{Exact eigenpairs and stiffness reduction in 1D} \label{sec:sr}

We focus on softIGA with uniform elements in 1D to establish sharp values for $\eta_{\max}$; then, using tensor products, we generalize the 1D analytical results to multidimensional problems by following the derivations in~\cite{ calo2019dispersion}. The $C^0$-linear softIGA elements are identical to the linear softFEM. We refer to~\cite[\S 3.1]{ deng2021softfem} for the exact eigenpairs and the asymptotic stiffness reduction ratio and percentage ($\rho_1 = \frac{3}{2}$ and $\varrho_1 = 33.3 \%$ respectively).  We distinguish the stiffness and mass matrices for the $p$-th order elements with a subscript $p$; for example, we use $\tilde{K}_p$ to denote the $p$-th order OF-IGA stiffness matrix.

\subsection{$C^1$-quadratic elements}

For $C^1$-quadratic splines with $N$ uniform elements, it is well-known that the OF-IGA (it is the same as the IGA in this case) bilinear forms $a(\cdot, \cdot)$ and $b(\cdot, \cdot)$  lead to the following stiffness and mass matrices (see, for example, \cite[\S~A.2]{ hughes2014finite}):
\begin{equation}
\tilde{K}_2  = \frac{1}{h}
\begin{bmatrix}
  \frac{4}{3} & -\frac{1}{6} & -\frac{1}{6} \\[0.2cm]
  -\frac{1}{6} & 1 & -\frac{1}{3} & -\frac{1}{6} \\[0.2cm]
  -\frac{1}{6} & -\frac{1}{3} & 1 & -\frac{1}{3} & -\frac{1}{6} \\[0.2cm]
  & \ddots & \ddots & \ddots & \ddots & \ddots &  \\[0.2cm]
  &  & -\frac{1}{6} & -\frac{1}{3} & 1 & -\frac{1}{6}  \\[0.2cm]
  && & -\frac{1}{6} & -\frac{1}{6} & \frac{4}{3}  \\
\end{bmatrix}, 
\tilde{M}_2 = h
\begin{bmatrix}
  \frac{1}{3} & \frac{5}{24} & \frac{1}{120} \\[0.2cm]
  \frac{5}{24} & \frac{11}{20} & \frac{13}{60} & \frac{1}{120} \\[0.2cm]
  \frac{1}{120} & \frac{13}{60} & \frac{11}{20} & \frac{13}{60} & \frac{1}{120} \\[0.2cm]
  & \ddots & \ddots & \ddots & \ddots & \ddots &  \\[0.2cm]
  &  & \frac{1}{120} & \frac{13}{60} & \frac{11}{20} & \frac{5}{24} \\[0.2cm]
  && & \frac{1}{120} & \frac{5}{24} & \frac{1}{3}  \\
\end{bmatrix},
\end{equation}
which are of dimension $N\times N$.  The softness bilinear form $s(\cdot, \cdot)$ leads to the matrix
\begin{equation}
  S_2 = 
  \begin{bmatrix}
    35 & -21 & 7 & -1 \\
    -21 & 21 & - 15 & 6 & -1 \\
    7 & -15 & 20 & -15 & 6 & -1 \\
    -1 & 6 & -15 & 20 & -15 & 6 & -1 \\
    & \ddots & \ddots & \ddots & \ddots & \ddots & \ddots & \ddots \\
  \end{bmatrix}_{N\times N},
\end{equation}
where the entries near the right boundary are such that the matrix is symmetric and persymmetric.  Also, these matrices are Toeplitz-plus-Hankel matrices, and the analytical eigenpairs can be derived by following \cite[\S~2.2]{ deng2021analytical}.

\begin{lemma}[Analytical eigenvalues and eigenvectors, $p=2$] \label{lem:epp2}
  For $C^1$ quadratic softIGA with $N$ uniform elements on $[0,1]$.  The GMEVP~\eqref{eq:mevppp} is $(\hat K_2 - \eta S_2) \hat U = \hat \lambda^h \tilde M_2 \hat U.$ Its eigenpairs are $(\hat{\lambda}_j^h, \hat{U}_j)$ for all $j\in\{1,\ldots,N\}$ where
  \begin{equation} \label{eq:epp2}
    \begin{aligned}
      \hat\lambda_j^h & = \frac{80\sin^2(\frac{t_j}{2})}{h^2}\frac{2 - 18\eta + (1+24\eta) \cos(t_j) -6 \eta \cos(2t_j)}{33 + 26\cos(t_j) + \cos(2t_j) },  \\
      \hat{U}_{j} & = c_j \big(\sin((k-1/2)t_j)\big)_{k \in\{1,\ldots, N\}}
    \end{aligned}
  \end{equation}
  with $t_j:=j \pi h$ and some normalization constant $c_j> 0$.
\end{lemma}
\begin{proof}
  The result~\eqref{eq:epp2} follows from an application of \cite[Thm.~2.2]{deng2021analytical}.
\end{proof}

\begin{remark}[Sharpness of $\eta_{\max}$, $p=2$]
  The coercivity of the softIGA bilinear form requires all the eigenvalues to be positive. Thus, 
  $$
  \hat\lambda_j^h  = \frac{80\sin^2(\frac{t_j}{2})}{h^2}\frac{2 - 18\eta + (1+24\eta) \cos(t_j) -6 \eta \cos(2t_j)}{33 + 26\cos(t_j) + \cos(2t_j) } > 0
  $$
  for all $t_j \in(0, \pi]$, which boils down to finding $\eta$ such that
  $$
  \eta  < \frac{2+\cos(t_j)}{48\sin^4(\frac{t_j}{2})}
  $$
  for all $t_j \in(0, \pi]$. Thus, the minimum of the right-hand-side occurs when $t_j = \pi$. Thus, the final condition is $\eta < \frac{1}{48}.$ This coincides with $\eta_{\max} = \frac{1}{2C^2_{2,3}}$ as shown in Theorem~\ref{thm:coe}. Therefore, the value $\eta_{\max} = \frac{1}{48}$ is sharp.
\end{remark}

\begin{figure}[h!]
\includegraphics[height=5.2cm]{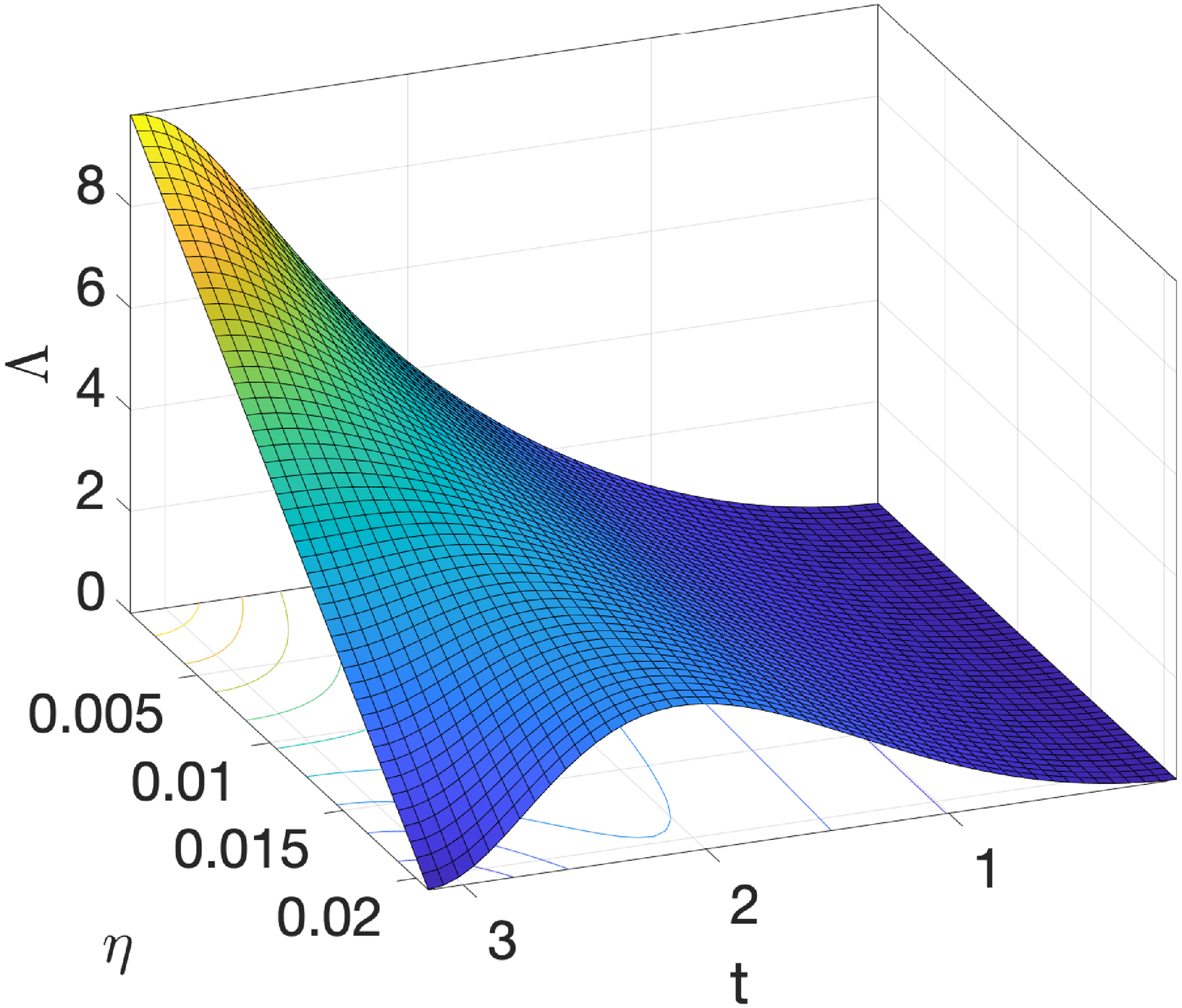} 
\includegraphics[height=5cm]{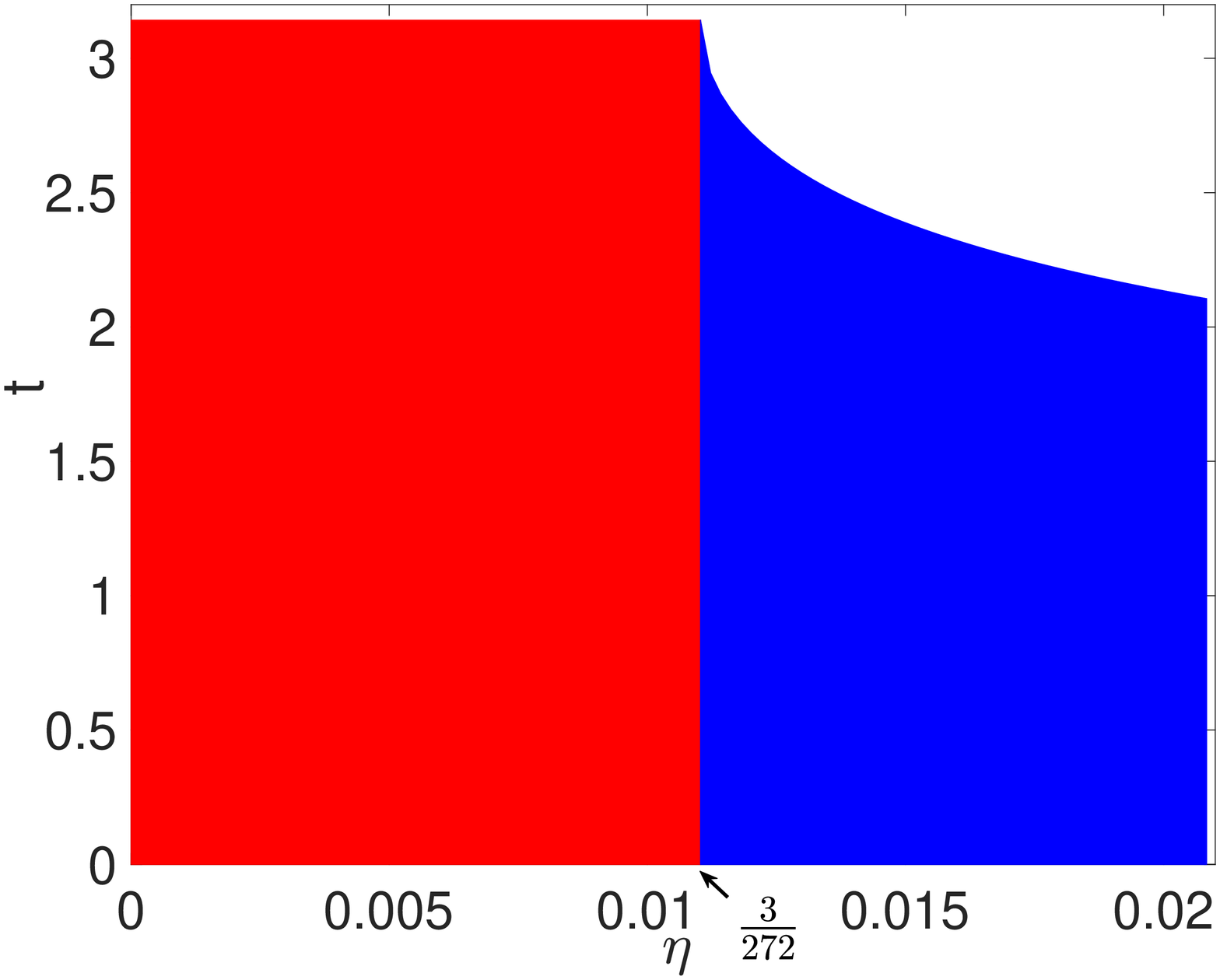} 
\caption{$C^1$-quadratic softIGA. Scaled eigenvalues $\Lambda = \hat\lambda_j^h h^2$ with respect to scaled eigenfrequency $t = j \pi h$ and $\eta$ (left plot) and the region where the approximate eigenvalue increases with respect to mode index (right plot).}
\label{fig:eigp2}
\end{figure}

We analyze the stiffness reduction by showing that the choice of $\eta$ impacts the eigenvalue ordering. For simplicity, we denote $\Lambda = \hat\lambda_j^h h^2$ and $t = t_j \in(0, \pi]$. Let $\eta \in [0, \frac{1}{48})$. Figure~\ref{fig:eigp2} shows the impact of the softness parameter $\eta$ on the distribution of the softIGA approximated eigenvalues.  The monotonicity with respect to $t$ changes for different values of $\eta$.  With this in mind, we choose the softness parameter as follows.

\begin{remark}[Choice of $\eta$, $p=2$]
  We first point out that the analytical eigenvectors given in Lemma~\ref{lem:epp2} do not depend on the softness parameter $\eta$.  The eigenvalues are monotonically increasing while $\eta \in (0, \frac{3}{272}]$.  For $\eta \in (\frac{3}{272}, \frac{1}{48}]$, the approximate eigenvalues are first increasing then decreasing, which is non-physical.  In physics, the modes with more oscillations have larger frequencies and eigenvalues.  For the approximate eigenvalues to be listed in ascending order and paired with the exact eigenvalues, we choose $\eta \in (0, \frac{3}{272}]$.  by default, for $C^1$-quadratic elements, we choose $\eta = \frac{3}{272}$.
\end{remark}

With the analytical eigenvalues given in~\eqref{eq:epp2}, one can derive the stiffness reduction ratio and the asymptotic stiffness reduction ratio and obtain
\begin{equation} 
  \rho^h_2 = \frac{2-18\eta + (1+24\eta)\cos(\pi h) - 6 \eta \cos(2\pi h)}{(1 - 48\eta )(2+ \cos(\pi h) },
  \qquad
  \rho_2 = \lim_{h \to 0} \rho^h_2 = \frac{1}{1-48\eta},
\end{equation}
where $\eta \in (0, \frac{3}{272}]$.  It is trivial to see that the ratio $\rho^h_2$ is increasing with respect to $\eta$ with $\eta \in (0, \frac{3}{272}]$.  Thus, the asymptotic maximum reduction ratio is when $\eta = \frac{3}{272}$:
\begin{equation} 
  \rho_2 = \frac{1}{1-48\cdot 3/272} = \frac{17}{8},
\end{equation}
 which leads to an asymptotic reduction percentage of 
\begin{equation} 
\varrho_2 = 100(1 - 1/\rho_2) \% \approx 52.9\%.
\end{equation}

\begin{theorem}[Eigenvalue optimal convergence and superconvergence, $p=2$] \label{thm:p2}
  Let $\lambda_j$ be the $j$-th exact eigenvalue of~\eqref{eq:pde} and let $\hat \lambda_j^h$ be the $j$-th approximate eigenvalue using $C^1$-quadratic softIGA with $N$ uniform elements on $[0,1]$.  Then the eigenvalue errors satisfy
  \begin{equation} \label{eq:errp2}
    \frac{ |\hat \lambda_j^h - \lambda_j|}{\lambda_j}  < \Big( \frac{37}{5,040} + \eta \Big) (j \pi h)^4, \qquad \forall j\in\{1,\ldots, N\}.
  \end{equation}
  Moreover, if $\eta = \frac{1}{720}$, the following holds:
  \begin{equation}
    \frac{ |\hat \lambda_j^h - \lambda_j|}{\lambda_j}  < \frac{1}{1,680} (j \pi h)^6, \qquad \forall j\in\{1,\ldots, N\}.
  \end{equation}
\end{theorem}

\begin{proof}
  The exact eigenvalues for~\eqref{eq:pde} with $\Omega=[0,1]$ are $\lambda_j = (j \pi)^2$ and the approximate eigenvalues $\hat \lambda_j^h$ are given in~\eqref{eq:epp2}.  Since $h=1/N$, for $j=N$, it is easy to verify both inequalities.  In the view of dispersion error analysis, applying a Taylor expansion to $\hat \lambda_j^h$, we obtain (recall that $t_j:=j\pi h$)
  \begin{equation} \label{eq:dep2}
    \frac{ \hat \lambda_j^h - \lambda_j}{ \lambda_j} = \Big( \frac{1}{720} - \eta \Big) t_j^4 + \frac{1}{3,360}t_j^6 + \frac{1}{86,400}t_j^8+ \mathcal{O}(t_j^{10}),
  \end{equation} 
  which leads to the desired inequalities for small $h$. More rigorously, we first show the first inequality for arbitrary $j$. We calculate
  \begin{equation*}
    \frac{ | \hat \lambda_j^h - \lambda_j |}{ \lambda_j} = \left| \frac{80\sin^2(\frac{t_j}{2})}{t_j^2 }\frac{2 - 18\eta + (1+24\eta) \cos(t_j) -6 \eta \cos(2t_j)}{33 + 26\cos(t_j) + \cos(2t_j) } -1 \right|.
  \end{equation*} 
  Since $t_j = j\pi h, j=1,2,\cdots, N,$ samples the interval $(0,\pi]$, we can consider a continuous variable $t\in (0,\pi)$ and prove more generally that
  \[
    \left| \frac{80\sin^2(\frac{t}{2})}{t^2 }\frac{2 - 18\eta + (1+24\eta) \cos(t) -6 \eta \cos(2t)}{33 + 26\cos(t) + \cos(2t) } -1 \right| < \Big( \frac{37}{5,040} + \eta \Big) t^4,
  \]
  which can be proved by analyzing the function monotonicity and extreme values.  Similarly, we perform the analysis to show the second inequality, which completes the proof.
\end{proof}

\begin{remark}[Constant sharpness, extra-order superconvergence, and dispersion error, $p=2$]
  We have the following observations.
  \begin{itemize}
  \item The constants in the inequalities in Theorem~\ref{thm:p2} are not sharp.  Smaller constants are possible, but the proof is more involved.

  \item If we add the softness bilinear form with parameter $\eta_b$ to the mass bilinear form in~\eqref{eq:vfhhh}, the resulting matrix eigenvalue problem is of the form $(\hat K_2 - \eta S_2) \hat U = \hat \lambda^h (\tilde M_2 + \eta_b S_2)\hat U.$ When $\eta = \frac{1}{720}, \eta_b = \frac{1}{3,360}$, the relative eigenvalue error becomes
    \begin{equation} \label{eq:dep2new}
      \frac{ \hat \lambda_j^h - \lambda_j}{ \lambda_j} = \frac{1}{86,400}t_j^8+ \mathcal{O}(t_j^{10}),
    \end{equation}  
    which leads to a superconvergence of order $8$ (4 extra orders than the optimal convergence case).  This extra-order superconvergence can also be obtained by generalizing the stiffness and mass entries; see~\cite{ idesman2018use, idesman2020new} for details.  The study on superconvergence is not the focus of this work and is subject to future work.

  \item With the approximate eigenvalues given in~\eqref{eq:epp2}, the dispersion errors can be written in a similar form of~\eqref{eq:dep2}.  We provide more details in Section~\ref{sec:de} for the dispersion error analysis for OF-IGA and softIGA.

  \end{itemize}
\end{remark}

\subsection{$C^2$-cubic element}

For $C^2$-cubic splines with $N$ uniform elements on $\Omega=[0,1]$, the OF-IGA bilinear forms $a(\cdot, \cdot)$ and $b(\cdot, \cdot)$  lead to the following stiffness and mass matrices:
\begin{equation} \label{eq:mk3}
  \begin{aligned}
    \tilde{K}_3 & = \frac{1}{h}
    \begin{bmatrix}
      \frac{13}{15} & -\frac{7}{60} & -\frac{1}{5} & -\frac{1}{120} \\[0.2cm]
      -\frac{7}{60} & \frac{2}{3} & -\frac{1}{8} & -\frac{1}{5} & -\frac{1}{120}  \\[0.2cm]
      -\frac{1}{5} & -\frac{1}{8} & \frac{2}{3} & -\frac{1}{8} & -\frac{1}{5} & -\frac{1}{120}  \\[0.2cm]
      -\frac{1}{120}  & -\frac{1}{5} & -\frac{1}{8} & \frac{2}{3} & -\frac{1}{8} & -\frac{1}{5} & -\frac{1}{120}  \\[0.2cm]
      & \ddots & \ddots & \ddots & \ddots & \ddots & \ddots  \\[0.2cm]
    \end{bmatrix}_{(N-1)\times (N-1)}, \\
    \tilde{M}_3 &= h
    \begin{bmatrix}
      \frac{41}{90} & \frac{17}{72} & \frac{1}{42} & \frac{1}{5040} \\[0.2cm]
      \frac{17}{72} & \frac{151}{315} & \frac{397}{1,680} & \frac{1}{42} & \frac{1}{5,040}  \\[0.2cm]
      \frac{1}{42} & \frac{397}{1,680} & \frac{151}{315} &\frac{397}{1,680} & \frac{1}{42} & \frac{1}{5,040}  \\[0.2cm]
      \frac{1}{5,040}  & \frac{1}{42} & \frac{397}{1,680} & \frac{151}{315} & \frac{397}{1,680} & \frac{1}{42} & \frac{1}{5,040}  \\[0.2cm]
      & \ddots & \ddots & \ddots & \ddots & \ddots & \ddots  \\[0.2cm]
    \end{bmatrix}_{(N-1)\times (N-1)}, 
  \end{aligned}
\end{equation}
where the entries near the right boundary are such that the matrices are symmetric and persymmetric.  The softness bilinear form $s(\cdot, \cdot)$ leads to the following matrix
\begin{equation}
  S_3 = 
  \begin{bmatrix}
    42 & -48 & 27 & -8 & 1 \\
    -48 & 69 & -56 & 28 & -8 & 1 \\
    27 & -56 & 70 & -56 & 28 & -8 & 1 \\
    -8 & 28 & -56 & 70 & -56 & 28 & -8 & 1 \\
    1 & -8 & 28 & -56 & 70 & -56 & 28 & -8 & 1 \\
    & \ddots & \ddots & \ddots & \ddots & \ddots & \ddots & \ddots \\
  \end{bmatrix}_{(N-1)\times (N-1)},
\end{equation}
where the entries near the right boundary are such that the matrix is symmetric and persymmetric.  As in the $C^1$-quadratic case, all these matrices are Toeplitz-plus-Hankel matrices, and the analytical eigenpairs can be derived by following~\cite[\S 2.2]{ deng2021analytical}.

\begin{lemma}[Analytical eigenvalues and eigenvectors, $p=3$] \label{lem:epp3}
  For $C^2$-cubic softIGA with $N$ uniform elements on $[0,1]$.  The GMEVP~\eqref{eq:mevppp} is $(\hat K_3 - \eta S_3) \hat U = \hat \lambda^h \tilde M_3 \hat U.$ Its eigenpairs are $(\hat{\lambda}_j^h, \hat{U}_j)$ for all $j\in\{1,\ldots,N-1\}$ where
  \begin{equation} \label{eq:epp3}
    \begin{aligned}
      \hat\lambda_j^h & = \frac{168\sin^2(\frac{t_j}{2})}{h^2}\frac{33 - 1,200 \eta + 2 (13 + 900 \eta) \cos(t_j) + (1- 720 \eta) \cos(2t_j) + 120 \eta \cos(3t_j)}{1,208 + 1,191\cos(t_j) + 120\cos(2t_j) + \cos(3t_j)},  \\
      \hat{U}_{j} & = c_j \big(\sin(k t_j)\big)_{k \in\{1,\ldots, N-1\}}
    \end{aligned}
  \end{equation}
  with $t_j:=j \pi h$ and some normalization constant $c_j> 0$.
\end{lemma}
\begin{proof}
  This can be proved by an application of~\cite[Thm.~2.1]{ deng2021analytical}.
\end{proof}

\begin{remark}[Sharpness of $\eta_{\max}$, $p=3$]
  Similarly as in the $C^1$-quadratic case, we solve for $\eta$ such that $ \hat\lambda_j^h > 0 $ for all $t_j \in(0, \pi]$, which boils down to finding $\eta$ such that
$$
\eta  < \frac{33+26\cos(t_j) + \cos(2t_j) }{3,840\sin^6(\frac{t_j}{2})}
$$
for all $t_j \in(0, \pi]$. The right-hand-side term is monotonically decreasing with respect to $t_j$ and the minimum is reached when $t_j = \pi$.  Thus, the condition for positive eigenvalue (coercivity) is $\eta < \frac{1}{480}.$
\end{remark}

\begin{figure}[h!]
\includegraphics[height=5.2cm]{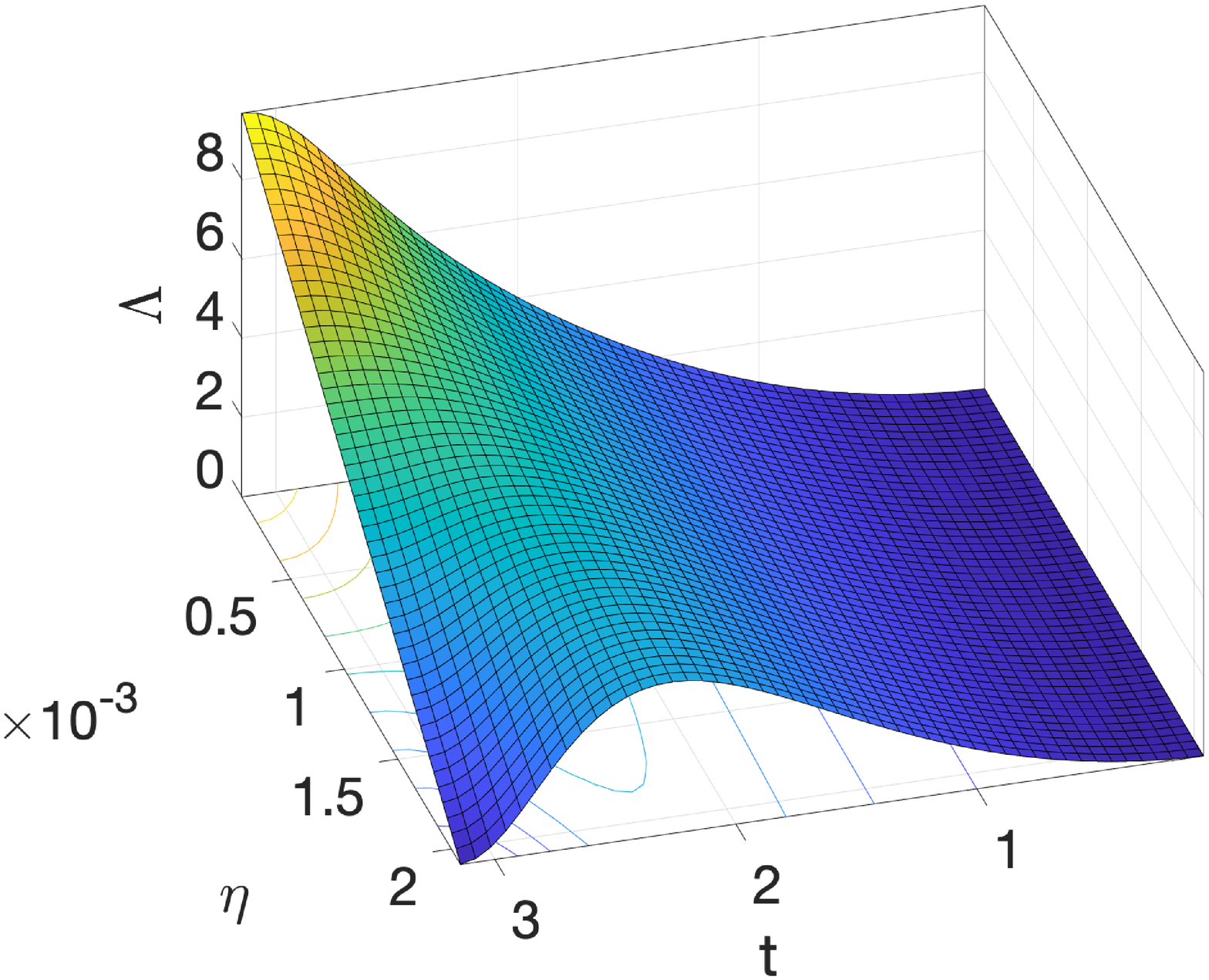} 
\includegraphics[height=5cm]{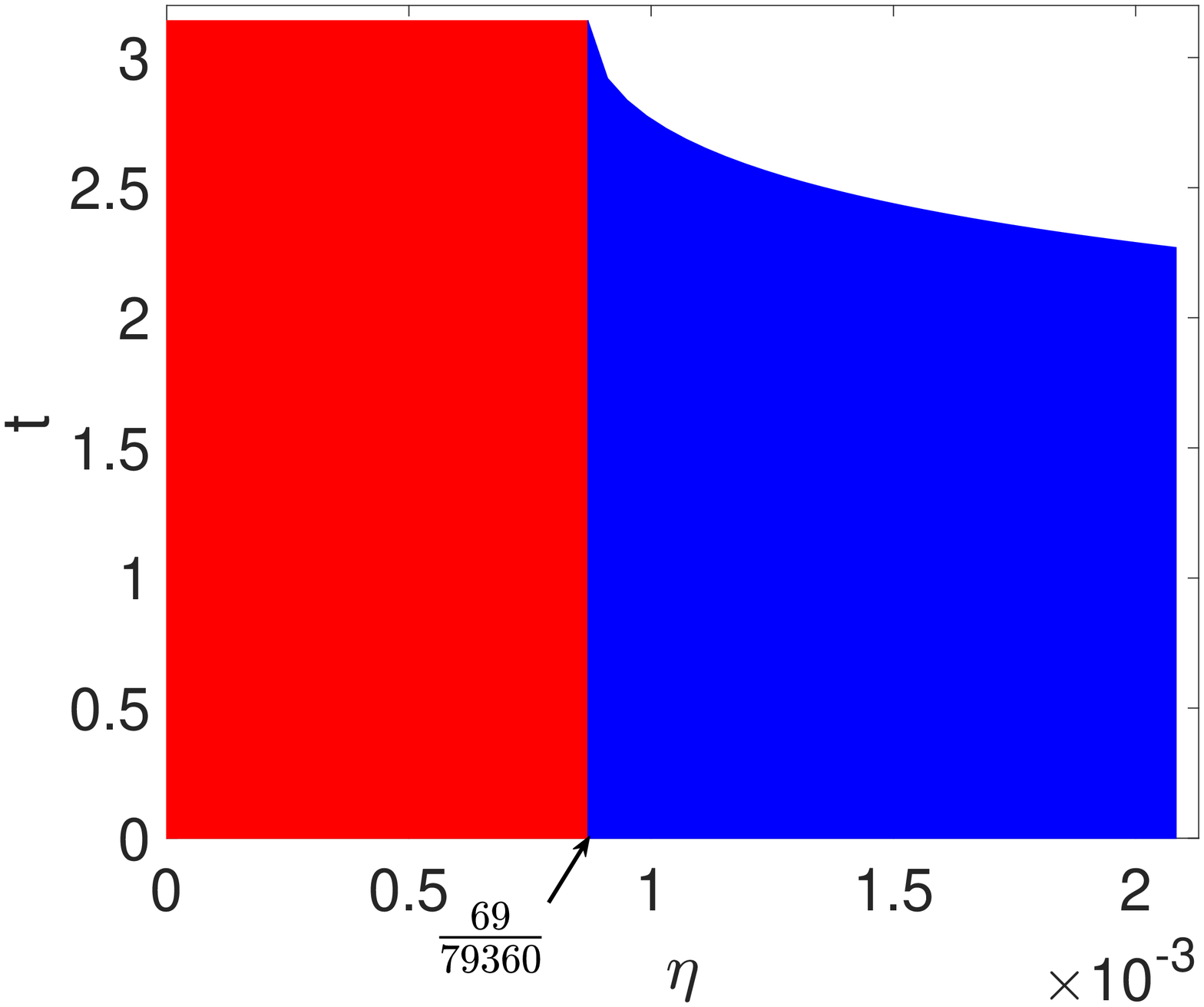} 
\caption{$C^2$-cubic softIGA. Scaled eigenvalues $\Lambda = \hat\lambda_j^h h^2$ with respect to scaled eigenfrequency $t = j \pi h$ and $\eta$ (left plot) and the region where the approximate eigenvalue increases with respect to mode index (right plot).}
\label{fig:eigp3}
\end{figure}

Similarly, the softness parameter impacts the monotonicity of the approximate eigenvalues. Figure~\ref{fig:eigp3} shows that the approximate eigenvalues are monotonically increasing when $\eta \in (0, \frac{69}{79,360}]$. For $\eta \in (\frac{69}{79,360}, \frac{1}{480}]$, the approximate eigenvalues are first increasing then decreasing, which is non-physical. We set $\eta \in (0, \frac{69}{79,360}]$.  By default, for maximal stiffness reduction, we choose $\eta = \frac{69}{79,360}$.  There are outliers in the spectra for cubic and higher-order elements, and exact matrix eigenvalues are unknown. The stiffness reduction ratio depends on the eigenvalues of IGA matrices. We present the reduction ratios numerically in Section~\ref{sec:num}.

\begin{theorem}[Eigenvalue optimal convergence and superconvergence, $p=3$] \label{thm:p3}
  Let $\lambda_j$ be the $j$-th exact eigenvalue of~\eqref{eq:pde} and let $\hat \lambda_j^h$ be the $j$-th approximate eigenvalue using $C^2$-cubic softIGA with $N$ uniform elements on $[0,1]$.  Then the eigenvalue errors satisfy
  \begin{equation} \label{eq:errp3}
    \frac{ |\hat \lambda_j^h - \lambda_j|}{\lambda_j}  < \Big( \frac{131}{332,640} + \eta \Big) (j \pi h)^6, \qquad \forall j\in\{1,\ldots, N-1\}.
  \end{equation}
  Moreover, if $\eta = \frac{1}{30,240}$. The following holds:
  \begin{equation}
    \frac{ |\hat \lambda_j^h - \lambda_j|}{\lambda_j}  < \frac{1}{27,720} (j \pi h)^8, \qquad \forall j\in\{1,\ldots, N-1 \}.
  \end{equation}
\end{theorem}

\begin{proof}
  We establish the inequalities by following the derivations for $C^1$ quadratics.
\end{proof}

\begin{remark}
  For $\eta = \frac{1}{30,240}, \eta_b = \frac{1}{60,480}$, we obtain four extra superconvergent orders on the eigenvalue errors, that is,
  \begin{equation}
\frac{ |\hat \lambda_j^h - \lambda_j|}{\lambda_j}  = \frac{1}{532,224} (j \pi h)^{10} + \mathcal{O}((j \pi h)^{12}).
\end{equation}
\end{remark}

\subsection{$C^3$-quartic and higher-oder elements}

For $C^3$-quartic splines with $N$ uniform elements on $\Omega=[0,1]$, we list the softIGA matrices in~\ref{app:p4}.  The internal entries of the softness matrix $S_p$ are parts of Yang-Hui's triangle (see, for example,~\cite{ weisstein2002crc}; also called Pascal's triangle) with the signs of the entries alternating.  This section presents the main results for $C^3$-quartic elements and discusses the extension to higher-order elements.
\begin{lemma}[Analytical eigenvalues and eigenvectors, $p=4$] \label{lem:epp4}
  For $C^3$-quartic softIGA with $N$ uniform elements on $[0,1]$.  The GMEVP~\eqref{eq:mevppp} is $(\hat K_4 - \eta S_4) \hat U = \hat \lambda^h \tilde M_4 \hat U.$ The eigenpairs are $(\hat{\lambda}_j^h, \hat{U}_j)$ for all $j\in\{1,\ldots,N\}$ where
  \begin{equation} \label{eq:epp4}
    \begin{aligned}
      \hat\lambda_j^h & = \frac{288\sin^2(\frac{t_j}{2})}{h^2}\frac{1,208 - 176,400\eta + 3(397 + 94,080\eta) \cos(t_j)  + \Lambda_\eta }{78,095 + 88,234 \cos(t_j)  + 14,608 \cos(2t_j) + 502 \cos(3t_j)  + \cos(4t_j)},  \\
      \Lambda_\eta & = 120(1-1,176\eta) \cos(2t_j) + (1+40,320\eta) \cos(3t_j)  - 5,040\eta \cos(4t_j), \\
      \hat{U}_{j} & = c_j \big(\sin((k-1/2)t_j)\big)_{k \in\{1,\ldots, N\}}
    \end{aligned}
  \end{equation}
  with $t_j:=j \pi h$ and some normalization constant $c_j> 0$.
\end{lemma}

The proof applies~\cite[Thm.~2.2]{deng2021analytical} for the $C^1$-quadratic element case.  From Lemma~\ref{lem:epp4}, for all eigenvalues to be positive (coercivity), one can derive the condition for the softness parameter $\eta \in (0, \eta_{\max})$ with $\eta_{\max} = \frac{17}{80,640}$.  The eigenvalue monotonicity behaves similarly as Figure~\ref{fig:eigp2} for $C^1$-quadratic elements or Figure~\ref{fig:eigp3} for $C^2$-cubic elements.  For $\eta \in (0, \frac{451}{6,191,360}]$, the eigenvalues are monotonically increasing. Similarly, we default to $\eta = \frac{451}{6,191,360}$ for $C^3$-quartic elements. The relative eigenvalue errors converge optimally
\begin{equation} \label{eq:errp4}
\frac{ |\hat \lambda_j^h - \lambda_j|}{\lambda_j}  = C_{\eta} (j \pi h)^{8} + \mathcal{O}((j \pi h)^{12}),
\end{equation}
where $C_{\eta}$ is a positive constant that is independent of mesh size $h$ and index $j$.  The superconvergence occurs when $\eta = \frac{1}{1,209,600}$.

For higher-order splines ($p\ge 5$) with uniform elements on $\Omega=[0,1]$, for odd $p$, the resulting matrices have the pattern of~\cite[Thm.~2.1]{ deng2021analytical} while for even $p$, the resulting matrices follow the pattern of~\cite[Thm.~2.2]{ deng2021analytical}.  
The analytical eigenpairs can then be derived accordingly.  
With the analytical eigenpairs, one can then derive the range of parameter $\eta\in [0, \eta_{\max})$ for coercivity. 
The default choice reduces the condition numbers and allows the eigenvalues to be sorted in an ascending order.  
In general, $\eta_{\max}$ depends on $p$ and continuity order $k$. 
The default $\eta$ and $\eta_{\max}$ decrease as $p$ increases.
Once we obtain the analytical eigenvalues, we analyze the positivity for the matrices and eigenvalue errors.  The dispersion errors can then be derived as a by-product. For the eigenfunctions, since the entries of the eigenvectors are exact values of the true solutions, one expects optimal eigenfunction errors as in~\eqref{eq:ee} due to the interpolation theory (cf.,~\cite{ ern2021, ciarlet1978finite}).

\section{Dispersion error analysis of softIGA} \label{sec:de}

In this section, we first introduce the matrix commutator and demonstrate that the matrix commutativity is required to justify the use of Bloch wave assumption~\cite{ bloch1928quantum, kittel2018introduction} for dispersion analysis.  Thus, the lack of commutating property in IGA creates outliers in the IGA element's spectra with $p\ge 3$.  We also establish the exact spectral errors in the view of dispersion analysis.  For this purpose, we use the duality unified spectral and dispersion analyses of~\cite{ hughes2008duality}.  The dispersion error relations and estimates have been established for internal matrix rows in \cite[Theorem 1]{ deng2018ddm} for arbitrary order $p$ (for $p=2,3,4$, see also~\cite{ hughes2014finite}).  Moreover, the dispersion errors for multiple dimensions can be established by using the tensor-product structure; see, for example,~\cite{ ainsworth2010optimally} for finite elements and~\cite{ calo2019dispersion} for isogeometric elements.  The outlier-free IGA adjusts the boundary nodes to remove the outliers and does not affect the internal nodes.  Therefore, we have unified dispersion error relations and estimates for the internal and boundary matrix rows in OF-IGA.

\subsection{Matrix commuting: justification of Bloch wave assumption} \label{sec:zcom}

Before we perform the dispersion analysis for outlier-free IGA elements, we first justify the exactness of the Bloch wave assumption.  This exactness eliminates the outliers in the outlier-free IGA setting.  Moreover, it serves as a base for the derivation of dispersion errors.  Consequently, this contributes to the derivation of the exact spectral error.

In 1D, the unit interval $\Omega=[0,1]$ is uniformly partitioned into $N$ elements with nodes $x_j = jh, j=0,1,\cdots,N$ and $h = 1/N$.  For an odd order $p$, let $I_h^1 = \{ 1,2,\cdots, N-1\}$ denotes an index set associated with the nodes $x_j = jh, j\in I_h^1$ while for an even order $p$, let $I_h^2 = \{ 1,2,\cdots, N\}$ denotes an index set associated with the nodes $x_j = (j-1/2)h, j\in I_h^2$.  There are $N-1$ nodes for odd $p$ and $N$ nodes for even $p$.  We justify the Bloch wave assumption for IGA elements by introducing the transformation $T_p$ defined below:

\begin{itemize}
\item For OF-IGA or softIGA elements with an odd order $p$, we  define $T_p$ as a matrix of dimension $(N-1)\times (N-1)$ such that the $j$-th column has entries $T_{kj}, k\in I_h^1$ where  $T_{kj}$ is  the $j$-th (new) B-spline basis function evaluated at nodes $x_k = kh, k\in I_h^1$;

\item For OF-IGA or softIGA elements with an even order $p$, we  define $T_p$ as a matrix of dimension $N\times N$ such that the $j$-th column has entries $T_{kj}, k\in I_h^2$ where  $T_{kj}$ is  the $j$-th (new) B-spline basis function evaluated at nodes $x_k = (k-1/2)h, k\in I_h^2$.

\end{itemize}

The transformation $T_p$ is an invertible matrix.  For an arbitrary function $f \in C^0([0,1])$, for $p$ being odd, let $F$ be a vector with entries $F_j = f(x_j) = f(jh), j\in I_h^1$.  Let $f_h = \sum_{j\in I_h^1} \hat{F}_j \phi_p^j$ be the linear combination of the new B-spline basis functions such that $f_h (x_j) = f(x_j), j\in I_h^1.$ Let $\hat F$ denote the vector with entries $\hat{F}_j, j\in I_h^1.$ Then, there holds
\begin{equation} \label{eq:Tp}
  F = T_p \hat F.
\end{equation}
Similarly, this property holds true for $p$ being even with the nodal evaluations at $x_j = (j-1/2)h, j\in I_h^2$.  For linear elements ($p=1$), $T_p$ is an identity matrix.  For $T_p, p=2,3,4,5$, we refer to the matrices listed in~\ref{app:tp}.  In general, we observe that
\begin{itemize}
\item $T_p$ is symmetric and persymmetric;
\item $T_p$ has a bandwidth of $\lfloor \frac{p}{2} \rfloor$.
\end{itemize}

Now, let $A$ and $B$ be square matrices of the same dimension and we define the commutator~\cite{ horn2012matrix} as:
\begin{equation} \label{eq:com}
  [A,B] = AB - BA.
\end{equation}
Matrices $A$ and $B$ commute if $[A,B] = 0$.  With this in mind, we have the following zero commutator properties.
\begin{lemma} \label{lem:kmt}
  Let $\hat{K}_p$ and $\hat{M}_p$ be the stiffness and mass matrices in~\eqref{eq:mevppp} corresponding to the $p$-th order softIGA~\eqref{eq:vfhhh} with $N$ uniform elements on $\Omega=[0,1]$.  For $p=1,2,3,4,5,$ there holds
\begin{equation} \label{eq:kmt}
[\hat{K}_p, T_p] = 0, \qquad [\hat{M}_p, T_p] = 0.
\end{equation}
\end{lemma}

\begin{proof}
  For $p=1$, these identities hold true as $T_1$ is an identity matrix.  For $p=2,3,4,5,$ the matrices $T_p, \hat{K}_p, \hat{M}_p$ are given in~\ref{app:tp}, Section~\ref{sec:sr}, and~\ref{app:p4}.  Simple matrix calculations lead to the desired results.
\end{proof}

\begin{remark}
  Lemma~\ref{lem:kmt} holds for OF-IGA matrices as they are special cases of softIGA when $\eta=0$.  Also, without proof, we conjecture that Lemma~\ref{lem:kmt} holds for arbitrary-order softIGA elements.  The proof relies on finding a recursive or generic formula for the matrices entries, which depends on a generic formula for the new B-spline basis function constructions near the boundary.
\end{remark}

We now justify the use of the Bloch wave (or sinusoidal in the particular setting of this paper) assumption for the eigenvector $U=(U_j)$ of the matrix eigenvalue problem~\eqref{eq:mevppp}.  First, in standard FEM with Lagrangian basis functions, each entry $U_j$ of an eigenvector $U$ represents the approximate nodal value of the exact solution since (1) the Lagrangian basis function is 1  at a node and is zero at all other nodes and (2) the entry of the eigenvector represents the approximate solution evaluated at a node.  Thus, for the dispersion analysis of a FEM, the Bloch wave assumption $U_j = e^{\iota \omega j h}$ is directly applied at the nodes $x_j = jh$; see, for example,~\cite{ ainsworth2004discrete, ainsworth2010optimally}.  In the IGA setting, the eigenvector $U$ represents the coefficients of the linear combination of B-splines.  The vector with nodal approximations are given as $\bar U = T_p U$.  We now establish the following result.
\begin{theorem} \label{thm:u}
  Let $\hat{K}_p$ and $\hat{M}_p$ be the stiffness and mass matrices in~\eqref{eq:mevppp} corresponding to the $p$-th order softIGA~\eqref{eq:vfhhh} with $N$ uniform elements on $\Omega=[0,1]$.  Let $(\lambda^h, U)$ be the solution to the eigenvalue problem~\eqref{eq:mevppp} and $\bar U = T_p U$.  Assuming~\eqref{eq:kmt} holds true for arbitrary order $p$, then $(\lambda^h, \bar U)$ is also a solution to the same eigenvalue problem~\eqref{eq:mevppp}.
\end{theorem}

\begin{proof}
  Multiplying $T_p$ to both sides
  $$\hat{K}_p U = \lambda^h \hat{M}_pU$$
  from the left, we have 
  $$ T_p K_pU = \lambda^h T_p M_pU.$$
  Applying the zero commutator property~\eqref{eq:kmt} implies the desired result. 
\end{proof}

Theorem~\ref{thm:u} justifies the use of the Bloch wave assumption for the dispersion analysis of IGA-related methods.  For the standard IGA method with $p=1,2,$ the zero commutator property~\eqref{eq:kmt} holds, which justifies the Bloch wave assumption that consequently leads to a unified dispersion error for all matrix rows (internal and boundary rows).  For the standard IGA method with $p>2$, the zero commutator property~\eqref{eq:kmt} is no longer valid due to non-zero values appearing near the boundary elements.  The Bloch wave assumption is valid only for internal nodes in the sense of $N\to \infty$.  This non-exactness for the nodes associated with the boundary elements contributes to the outliers in the IGA spectra.  The exactness contributes to the elimination of the outliers.  In the softIGA/OF-IGA, the Bloch wave assumption is valid for all matrix rows.  Consequently, this unifies the dispersion errors and leads to outlier-free approximate spectra.

\subsection{Dispersion error analysis}

The dispersion analysis for $C^0$-linear and $C^1$-quadratic elements was established, and the dispersion errors for both internal and boundary rows are unified in~\cite{ hughes2008duality, hughes2014finite}. Nevertheless, the quadratic case is not explicitly unified in~\cite{ hughes2014finite}, but one can easily verify that the dispersion error equation (130) in~\cite{ hughes2014finite} holds for all the matrix rows. For the $C^2$-cubic case, the dispersion errors for the matrix rows of OF-IGA near the boundaries were established in~\cite{ deng2021boundary}. In this section, since OF-IGA is a special case of softIGA when $\eta=0$, we focus on the dispersion analysis for softIGA.

With the justification of Bloch wave assumption (see section~\ref{sec:zcom}) in mind, for an odd-order ($p=1,3,5,\cdots$) element, we assume that the component $U_{j,k}$ of the $j$-th eigenvector $U_j$ takes the Bloch waveform
\begin{equation} \label{eq:ujk}
  U_{j,k} = \sin(\omega_j k h), \quad k = 1, 2, \cdots, N-1,
\end{equation}
where $\omega_j = j\pi$ is the eigenfrequency associated with the $j$-th mode.  Similarly, for an even-order ($p=2,4,6,\cdots$) element, we assume that the component $U_{j,k}$ of the $j$-th eigenvector $U_j$ takes the Bloch waveform
\begin{equation} \label{eq:ujk2}
  U_{j,k} = \sin\big(\omega_j (k-\frac12) h \big), \quad k = 1, 2, \cdots, N.
\end{equation}

For $p$-th order element, the matrix rows start repeating from $l=p+1$ to $j=N-p-2$ for odd-valued $p$ while to $l=N-p-1$ for even-valued $p$. Those internal nodes are the same as in the standard IGA.  They have the dispersion relation~\cite[Proof of Lemma 1]{ deng2018ddm}
\begin{equation} \label{eq:dr}
  \hat{\lambda}_j^h h^2 = \frac{ \hat{K}_{l,l} + 2 \sum_{k=1}^p \hat{K}_{l,l+k} \cos(k \omega_j h)  }{\hat{M}_{l,l} + 2 \sum_{k=1}^p \hat{M}_{l,l+k} \cos(k \omega_j h) },
\end{equation}
which can be derived from the matrix problem~\eqref{eq:mevppp} using~\eqref{eq:ujk} for odd-valued $p$ or~\eqref{eq:ujk2} for even-valued $p$ with the trigonometric identity $\sin(\alpha \pm \beta) = \sin(\alpha) \cos(\beta) \pm \cos(\alpha) \sin(\beta)$.  Conventionally, the dispersion relation is written in terms of $\omega_j$ and $\omega_j^h := \sqrt{\hat{\lambda}_j^h}$ and the dispersion error is defined as $| \omega_j - \omega_j^h |$.  Herein, we equivalently represent the dispersion error in terms of the eigenvalues.  In softIGA, the matrices entries near the boundaries are regularized such that the dispersion relation~\eqref{eq:dr} is unified for all the rows of the matrix eigenvalue problem~\eqref{eq:mevppp}.  This unification is guaranteed by the resulting matrices patterns~\cite[Theorems 2.1 \& 2.2]{ deng2021analytical}.  In particular, with eigenfrequencies $\omega_j = j\pi$,~\eqref{eq:dr} simplifies to the analytical eigenvalues~\eqref{eq:epp2},~\eqref{eq:epp3}, and~\eqref{eq:epp4} for quadratic, cubic, and quartic softIGA elements, respectively.  With the matrices provided in Section~\ref{sec:sr} and appendices, one can easily derive the dispersion (squared relative) errors
\begin{equation} \label{eq:des}
\frac{(\omega^h_j)^2 - \omega_j^2}{\omega_j^2} = 
  \begin{cases}
  \Big( \frac{1}{720} - \eta \Big) (\omega_j  h)^4 + \frac{(\omega_j  h)^6}{3,360} + \mathcal{O}((\omega_j h)^8 ), & p=2, \\
  \Big( \frac{1}{30,240} - \eta \Big) (\omega_j  h)^6 + \frac{(\omega_j  h)^8}{60,480} + \mathcal{O}((\omega_j h)^{10} ), & p=3, \\
  \Big( \frac{1}{1,209,600} - \eta \Big) (\omega_j  h)^8 + \frac{(\omega_j  h)^{10}}{1,368,576} + \mathcal{O}((\omega_j h)^{12} ), & p=4, \\
  \Big( \frac{1}{47,900,160} - \eta \Big) (\omega_j  h)^{10} + \frac{691(\omega_j  h)^{12}}{24,216,192,000} + \mathcal{O}((\omega_j h)^{14} ), & p=5.
  \end{cases}
\end{equation}    
These dispersion errors lead to the optimal eigenvalue errors
\begin{equation}
\frac{ |\hat \lambda_j^h - \lambda_j|}{\lambda_j}  \le Ch^{2p}, 
\end{equation} 
where $C>$ is a constant independent of $h$. 
With the choices of the softness parameter
\begin{equation} 
\eta = 
  \begin{cases}
  \frac{1}{720}, & p=2, \\
  \frac{1}{30,240}, & p=3, \\
  \frac{1}{1,209,600}, & p=4, \\
  \frac{1}{47,900,160}, & p=5,
  \end{cases}
\end{equation}   
we have a superconvergent eigenvalue error
\begin{equation}
\frac{ |\hat \lambda_j^h - \lambda_j|}{\lambda_j}  \le Ch^{2p+2},
\end{equation} 
for $p=2,3,4,5$. 
We expect this to hold for higher-order elements.
Lastly, we note that with $\eta=0$, ~\eqref{eq:des} reduces to the dispersion errors of OF-IGA for $p=2,3,4,5.$

\section{Numerical experiments} \label{sec:num}

This section presents numerical simulations of problem~\eqref{eq:pde} in 1D, 2D, and 3D using isogeometric and tensor-product elements.  Once the eigenvalue problem is solved, we sort the discrete eigenpairs and pair them with the exact eigenpairs. We focus on the numerical approximation properties of the eigenvalues. However, in 1D, we also report the eigenfunction (EF) errors in $H^1$ semi-norm (energy norm). To have comparable scales, we collect the relative eigenvalue errors and scale the energy norm by the corresponding eigenvalues~\cite{ puzyrev2017dispersion, calo2019dispersion}. For softIGA approximations.
 we define the relative eigenvalue error as
\begin{equation}
\hat e_j = \frac{| \hat \lambda^h_j - \lambda_j | }{\lambda_j}
\end{equation}
Similarly, we define the eigenfunction errors in the $H^1$-seminorm (or energy) norm as $\hat e_{u_j} =  |u_j - \hat u_j^h|_{1,\Omega} $.

\begin{figure}[h!]
\includegraphics[height=5.8cm]{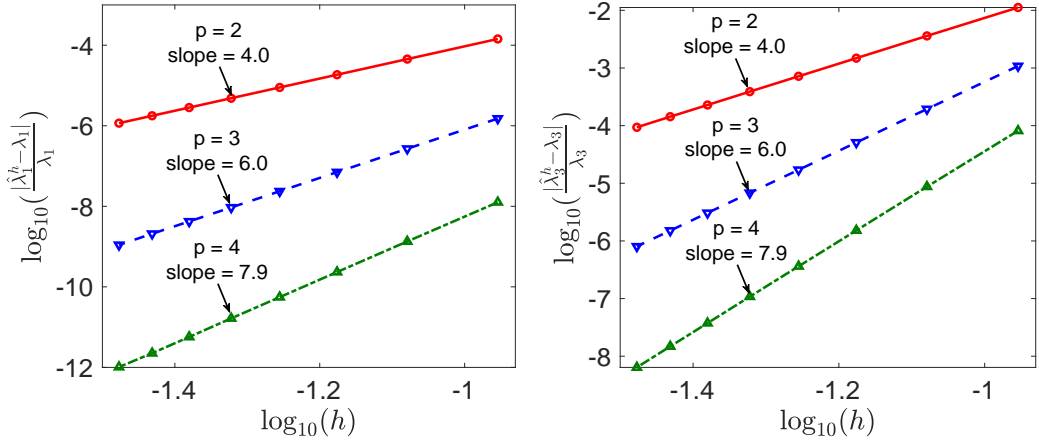} 
\vspace{-0.6cm}
\caption{Relative eigenvalue errors of softIGA with respect to mesh size $h$ in 1D for $p=\{2,3,4\}$.}
\label{fig:l1l3err1d}
\end{figure}

\subsection{Error estimates}

We first demonstrate the optimal eigenvalue and eigenfunction errors of softIGA.  With a cube domain $\Omega=[0, 1]^d, d=1,2,3$, the differential eigenvalue problem~\eqref{eq:pde} has true eigenvalues and eigenfunctions
\begin{equation*}
\begin{cases}
\lambda_j  = j^2 \pi^2, \quad \text{and} \quad u_j = \sqrt{2} \sin( j\pi x), \quad j = 1, 2, \cdots, & d=1, \\
 \lambda_{jk}  = ( j^2 + k^2 ) \pi^2, u_{jk} = \sin( j\pi x)\sin( k\pi y) \big),  j,k = 1, 2, \cdots, & d=2, \\
 \lambda_{jkl}  = ( j^2 + k^2 + l^2) \pi^2, u_{jkl} = \sin( j\pi x)\sin( k\pi y) \sin( l\pi z) \big),  j,k,l = 1, 2, \cdots, & d=3,
\end{cases}
\end{equation*}
respectively. 

\begin{figure}[h!]
\includegraphics[height=5.8cm]{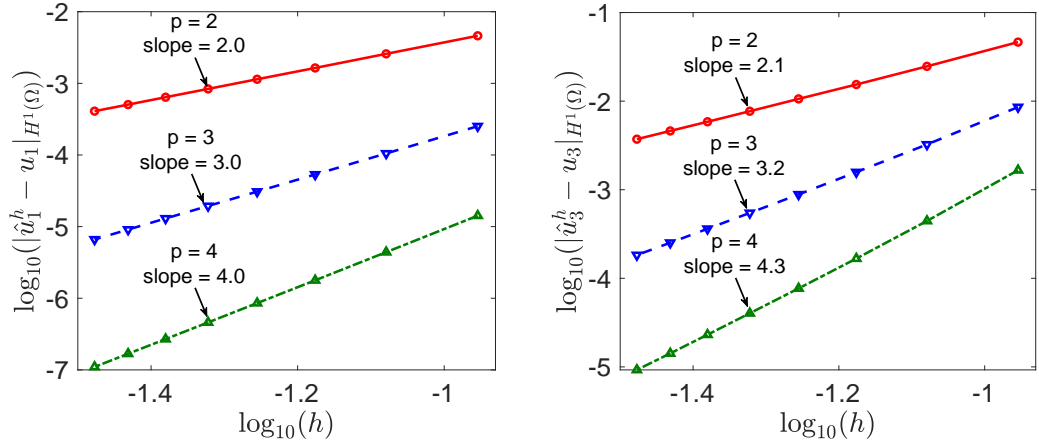} 
\vspace{-0.6cm}
\caption{Eigenfunction $H^1$-seminorm errors of softIGA with respect to mesh size $h$ in 1D for $p=\{2,3,4\}$.}
\label{fig:u1u3err1d}
\end{figure}

Figure~\ref{fig:l1l3err1d} shows the relative eigenvalue errors of softIGA for problem~\ref{eq:pde} in 1D with respect to the mesh size $h$.  There are $N = 3j, j =3,4, \cdots, 10,$ uniform elements.  We focus on the first and third eigenvalues and consider $p=2,3,4$.  By default, we set the softness parameter $\eta = \frac{3}{272}, \frac{69}{79,360}, \frac{451}{6,191,360}$ in~\eqref{eq:vfhhh} for $p=2,3,4,$ respectively.  The eigenvalue errors converge with rates $h^{2p}$, which confirms the optimal convergence~\eqref{eq:ee} for the eigenvalues.  The errors for other eigenvalues behave similarly.

Figure~\ref{fig:u1u3err1d} shows the $H^1$-seminorm errors of the first and third eigenfunctions approximated by softIGA elements.  The setting is the same as in Figure~\ref{fig:l1l3err1d}.  The errors are convergent of rates $h^{p}$, which confirms the optimal convergence~\eqref{eq:ee} for the eigenfunctions.  As  inequalities~\eqref{eq:errp2}, ~\eqref{eq:errp3}, and~\eqref{eq:errp4} show, the optimal convergence orders also hold true for the eigenvalue indices.  Figure~\ref{fig:eveferrn321d} shows the relative eigenvalue and eigenfunction $H^1$-seminorm errors with respect to eigenvalue index $j$.  We fix the mesh with $N=32$ uniform elements, and the errors are on a logarithmic scale.

\begin{figure}[h!]
\includegraphics[height=5.8cm]{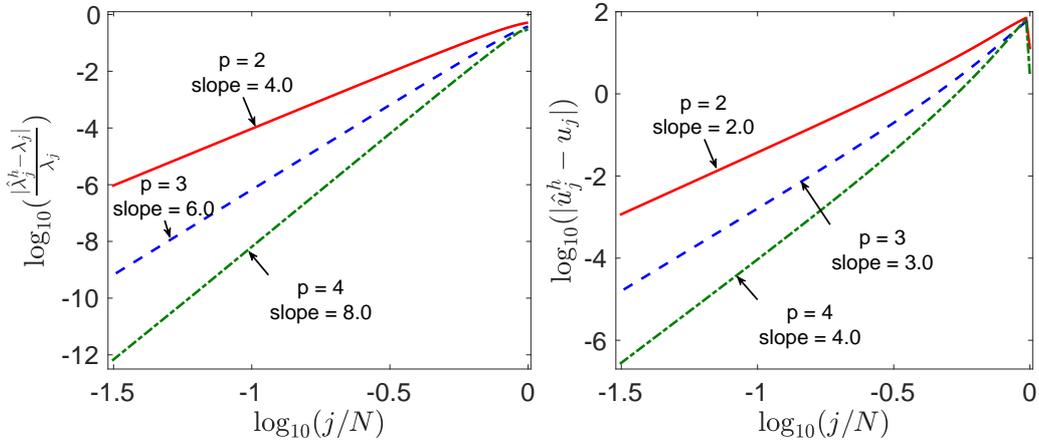} 
\vspace{-0.6cm}
\caption{Relative eigenvalue and eigenfunction $H^1$-seminorm errors in logarithmic scale with respect to eigenvalue index $j$ for $p=\{2,3,4\}$ with $N=32$ uniform softIGA elements in 1D.}
\label{fig:eveferrn321d}
\end{figure}

\begin{figure}[h!]
\includegraphics[height=5.8cm]{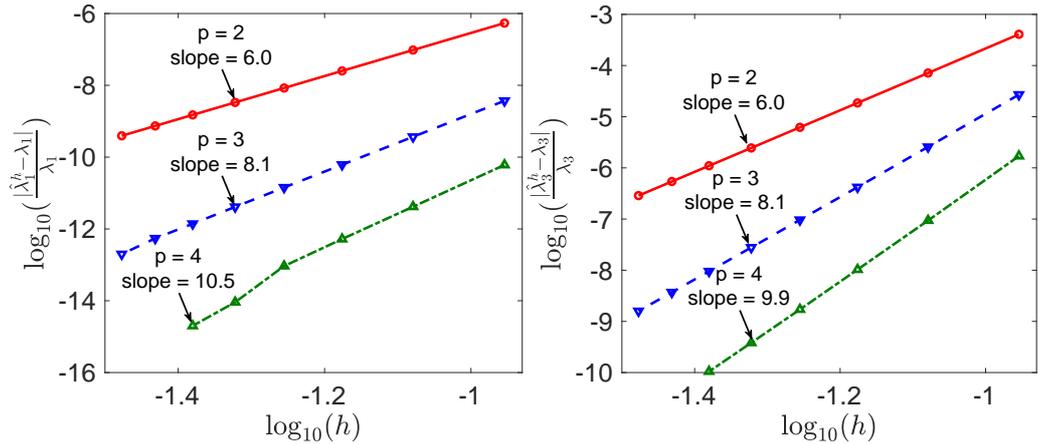} 
\vspace{-0.6cm}
\caption{Superconvergent relative eigenvalue errors of softIGA with respect to mesh size $h$ in 1D for $p=\{2,3,4\}$.}
\label{fig:l1l3superr1d}
\end{figure}

Figure~\ref{fig:l1l3superr1d} shows the superconvergence of order $h^{2p+2}$ for the relative eigenvalue errors of softIGA elements when choosing softness parameter $\eta = \frac{1}{720}, \frac{1}{30,240}, \frac{1}{1,209,600}$ in~\eqref{eq:vfhhh} for $p=2,3,4,$ respectively.  For $p=4$, the errors reach machine precision $10^{-15}$ when $N\approx 24.$ Thus, we omit the errors when $N=27,30$.  This superconvergence confirms the theoretical results established in Section~\ref{sec:sr}.
Lastly, Figures~\ref{fig:l1l3err2d} and~\ref{fig:l1l3err3d} show the relative eigenvalue errors in 2D and 3D, respectively.  Again, the errors have optimal convergence rates and confirm the theoretical expectations.

\begin{figure}[h!]
\includegraphics[height=5.8cm]{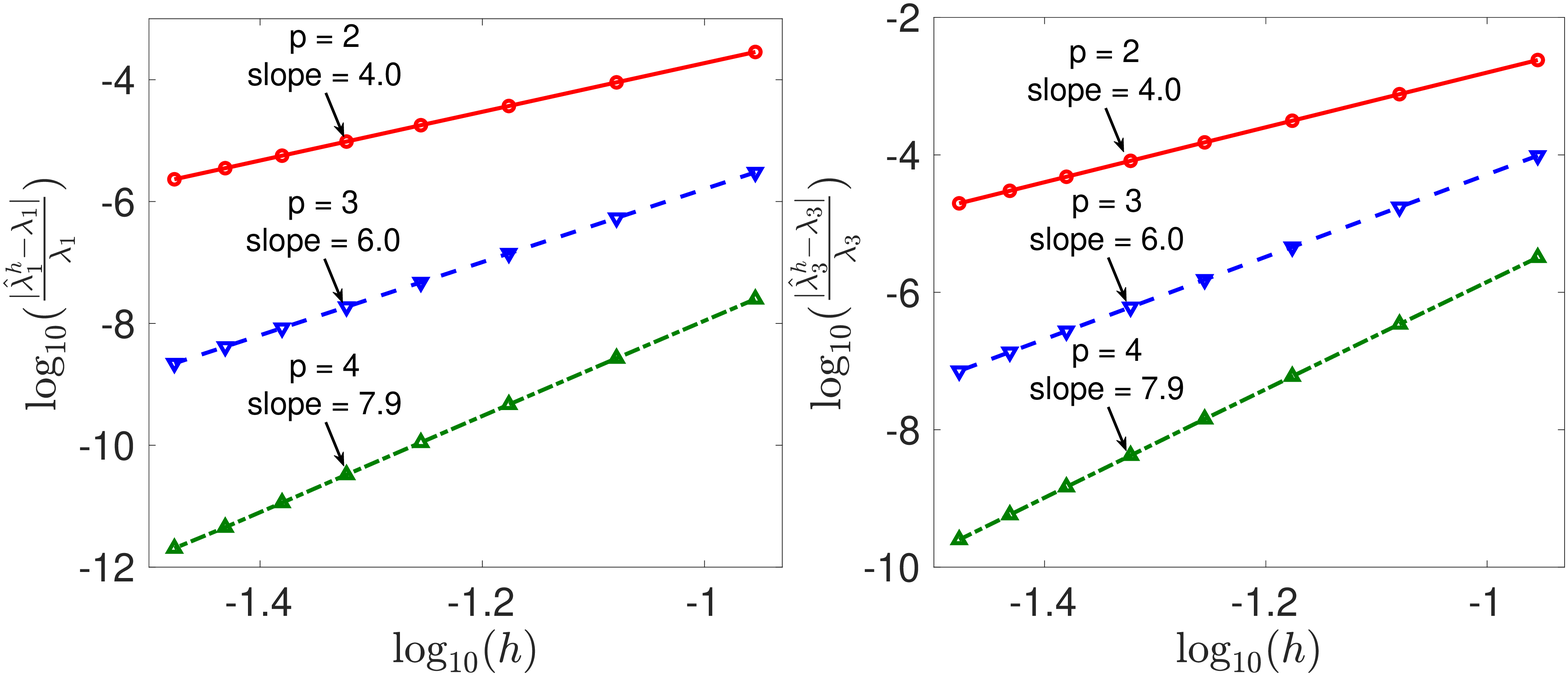} 
\vspace{-0.6cm}
\caption{Relative eigenvalue errors of softIGA with respect to mesh size $h$ in 2D for $p=\{2,3,4\}$.}
\label{fig:l1l3err2d}
\end{figure}

\begin{figure}[h!]
\includegraphics[height=5.8cm]{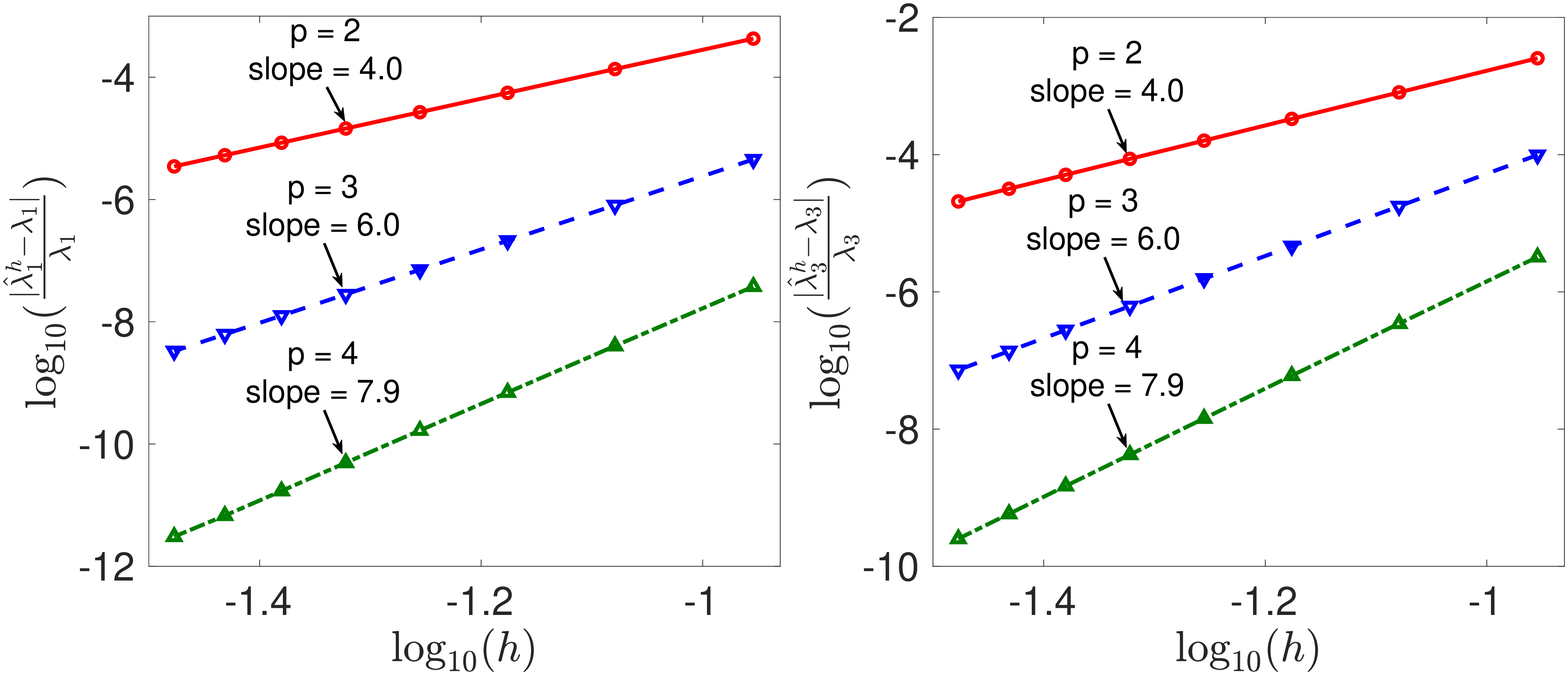} 
\vspace{-0.6cm}
\caption{Relative eigenvalue errors of softIGA with respect to mesh size $h$ in 3D for $p=\{2,3,4\}$.}
\label{fig:l1l3err3d}
\end{figure}

\subsection{Stiffness reduction}

\begin{table}[ht]
\centering 
\begin{tabular}{|c | c | ccc | ccc c| cc |}
\hline
$d$ & $p$ & $\lambda_{\min}^h$ &  $\lambda_{\max}^h$ & $ \hat \lambda_{\max}^h $ &  $\gamma$ &  $\hat \gamma$ & $\rho^h$ & $\varrho^h$ \\[0.1cm] \hline
 &	2 &	9.8696 &	1.0000e5 &	4.7059e4 &	1.0132e4 &	4.7681e3 &	2.1250 &	52.94\% \\[0.1cm]
1 &	3 &	9.8696 &	1.4556e5 &	5.7581e4 &	1.4748e4 &	5.8341e3 &	2.5279 &	60.44\% \\[0.1cm]
 &	4 &	9.8696 &	2.4490e5 &	6.8279e4 &	2.4814e4 &	6.9181e3 &	3.5868 &	72.12\% \\[0.1cm] \hline
 &	2 &	1.9739e1 &	3.2000e4 &	1.5059e4 &	1.6211e3 &	7.6289e2 &	2.1250 &	52.94\% \\[0.1cm]
2 &	3 &	1.9739e1 &	4.6579e4 &	1.8425e4 &	2.3597e3 &	9.3342e2 &	2.5280 &	60.44\% \\[0.1cm]
 &	4 &	1.9739e1 &	7.8369e4 &	2.1849e4 &	3.9702e3 &	1.1069e3 &	3.5868 &	72.12\% \\[0.1cm] \hline
 &	2 &	2.9609e1 &	1.2000e4 &	5.6471e3 &	4.0528e2 &	1.9072e2 &	2.1250 &	52.94\% \\[0.1cm]
3 &	3 &	2.9609e1 &	1.7470e4 &	6.9051e3 &	5.9004e2 &	2.3321e2 &	2.5301 &	60.48\% \\[0.1cm]
 &	4 &	2.9609e1 &	2.9392e4 &	8.1964e3 &	9.9268e2 &	2.7682e2 &	3.5860 &	72.11\% \\[0.1cm] \hline
 \end{tabular}
\caption{Minimal and maximal eigenvalues, condition numbers, stiffness reduction ratios, and reduction percentages when using IGA and softIGA with $p\in\{2,3,4\}$ and uniform meshes of $N=100, 40\times40, 20\times20\times20$ elements in 1D, 2D, and 3D, respectively.}
\label{tab:cond} 
\end{table}

\begin{table}[ht]
\centering 
\begin{tabular}{|c | c | ccc | ccc c| cc |}
\hline
$d$ & $p$ & $\tilde{\lambda}_{\min}^h$ &  $\tilde{\lambda}_{\max}^h$ & $ \hat \lambda_{\max}^h $ &  $\frac{\tilde{\lambda}_{\max}^h}{\tilde{\lambda}_{\min}^h}$ &  $\frac{\hat{\lambda}_{\max}^h}{\hat{\lambda}_{\max}^h}$ & $\tilde{\rho}^h$ & $\tilde{\varrho}^h$ \\[0.1cm] \hline
1 &3 & 9.8696 &   9.8675e4 &   5.7581e4 &   9.9979e3 &   5.8341e3 &   1.7137 &   41.65\% \\[0.1cm]
 &4 & 9.8696 &   9.8710e4 &   6.8279e4 &   1.0001e4 &   6.9181e3 &   1.4457 &   30.83\% \\[0.1cm] \hline
2 &3 & 1.9739e1 &   3.1331e4 &   1.8425e4 &   1.5872e3 &   9.3342e2 &   1.7004 &   41.19\% \\[0.1cm]
 &4 & 1.9739e1 &   3.1587e4 &   2.1849e4 &   1.6002e3  &  1.1069e3 &   1.4457 &   30.83\% \\[0.1cm] \hline
3 &3 & 2.9609e1 &   1.1437e4 &   6.9051e3 &   3.8627e2 &   2.3321e2 &   1.6563 &   39.63\% \\[0.1cm]
 &4 & 2.9609e1 &   1.1845e4 &   8.1964e3 &   4.0006e2 &   2.7682e2 &   1.4452 &   30.80\% \\[0.1cm] \hline
 \end{tabular}
\caption{Minimal and maximal eigenvalues, condition numbers, stiffness reduction ratios, and reduction percentages when using OF-IGA and softIGA with $p\in\{3,4\}$ and uniform meshes of $N=100, 40\times40, 20\times20\times20$ elements in 1D, 2D, and 3D, respectively.}
\label{tab:cond2} 
\end{table}

\begin{figure}[h!]
\includegraphics[height=6cm]{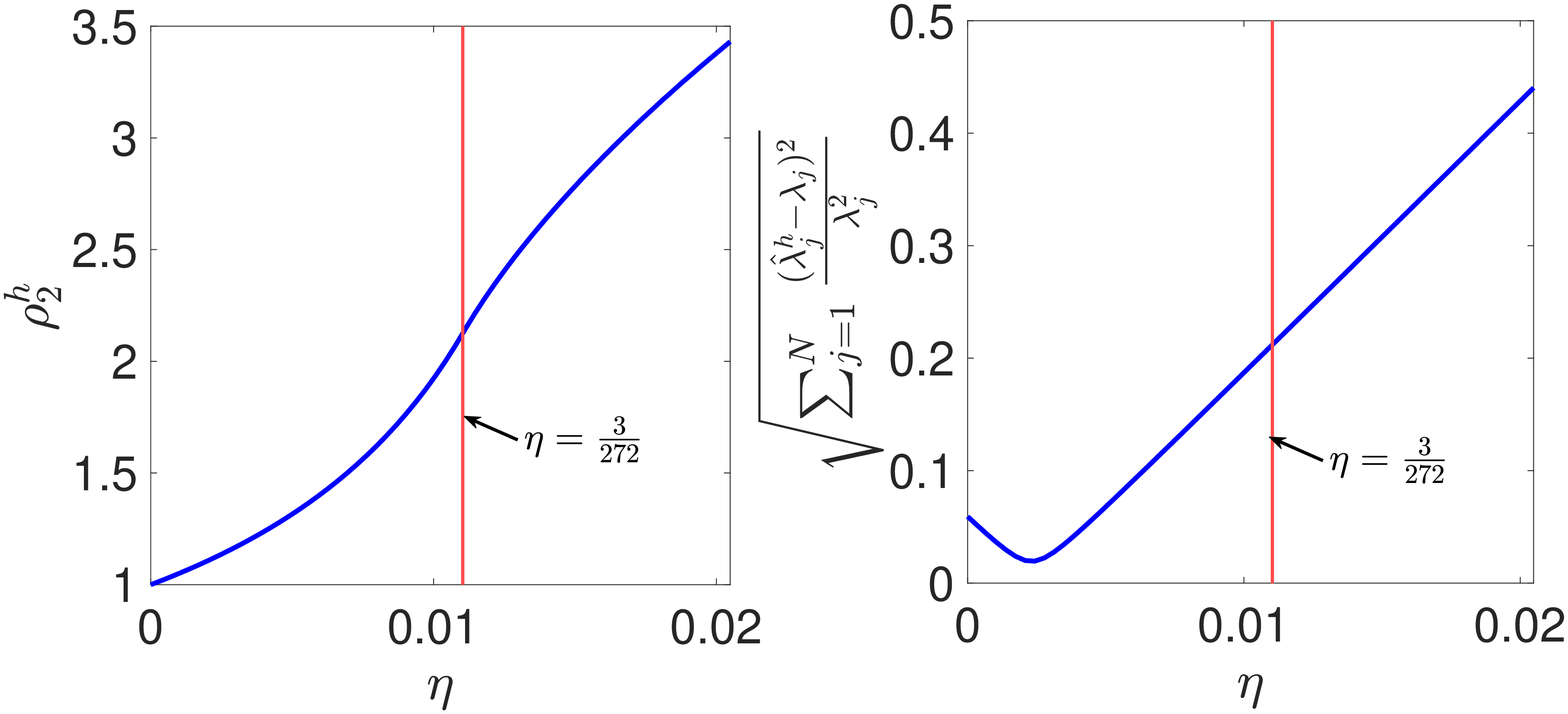} 
\vspace{-0.3cm}
\caption{Stiffness reduction ratios (left plot) and root mean square relative eigenvalue errors (right plot) with respect to softness parameter $\eta \in [0, \frac{1}{48}]$ for $C^1$-quadratic softIGA with $N=100$ uniform elements in 1D. }
\label{fig:etap2}
\end{figure}

\begin{figure}[h!]
\includegraphics[height=6cm]{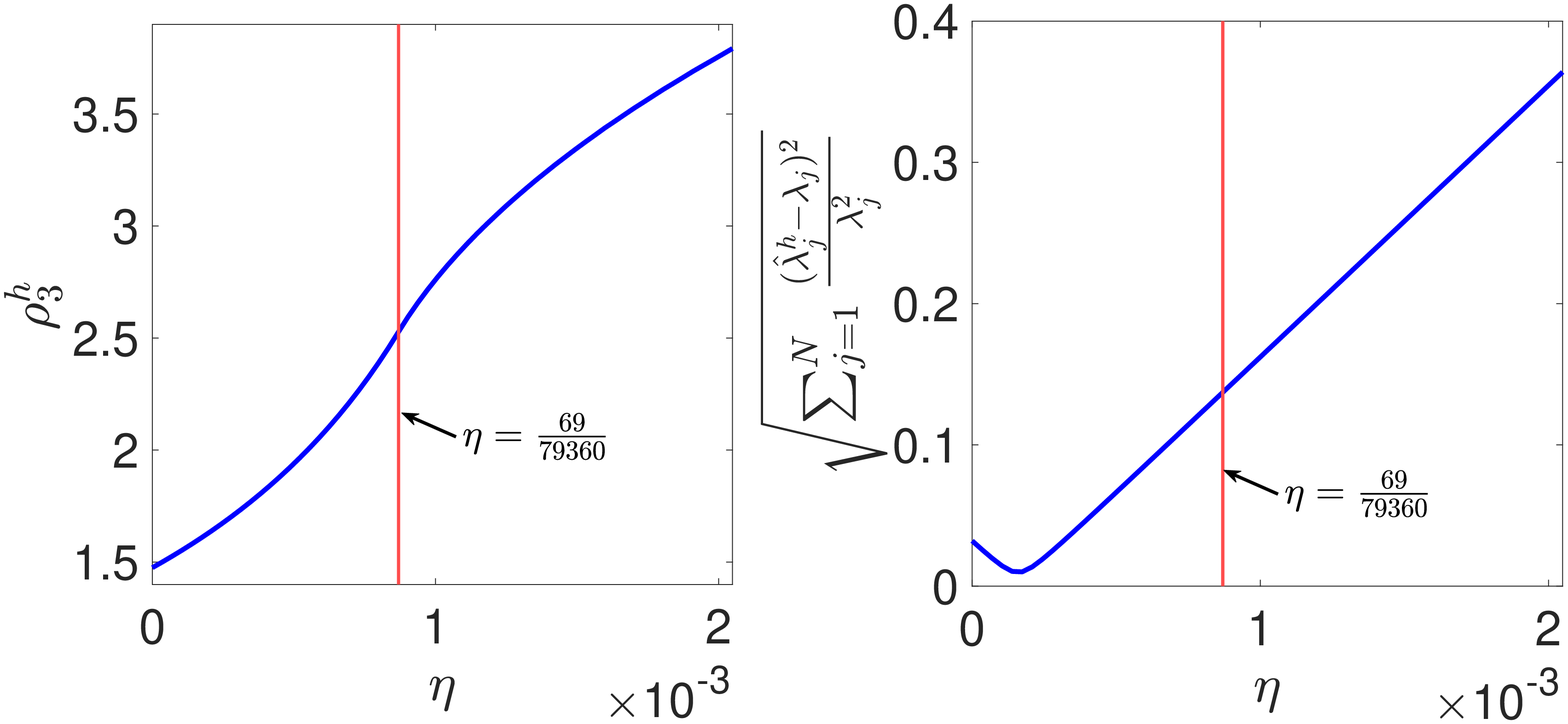} 
\vspace{-0.3cm}
\caption{Stiffness reduction ratios (left plot) and root mean square relative eigenvalue errors (right plot) with respect to softness parameter $\eta \in [0, \frac{1}{480}]$ for $C^2$-cubic softIGA with $N=100$ uniform elements in 1D. }
\label{fig:etap3}
\end{figure}

\begin{figure}[h!]
\includegraphics[height=6cm]{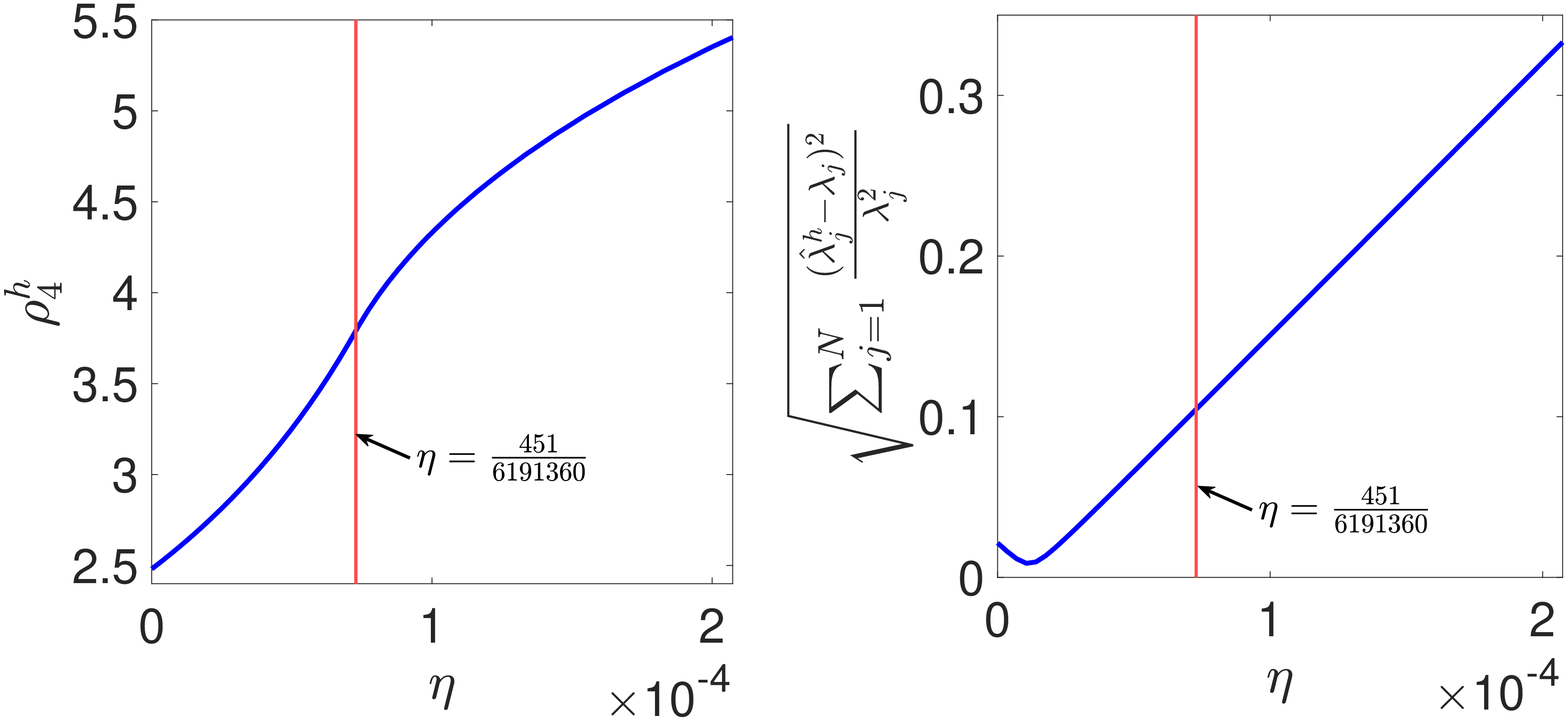} 
\vspace{-0.3cm}
\caption{Stiffness reduction ratios (left plot) and root mean square relative eigenvalue errors (right plot) with respect to softness parameter $\eta \in [0, \frac{17}{80,640}]$ for $C^3$-quartic softIGA with $N=100$ uniform elements in 1D. }
\label{fig:etap4}
\end{figure}

In this section, we study the stiffness of the softIGA systems and present the stiffness reduction ratios and percentages with respect to the standard IGA.  We define the stiffness (or condition number), reduction ratio, and reduction percentage in Section~\ref{sec:s}.  Table~\ref{tab:cond} shows the minimal and maximal eigenvalues, condition numbers, stiffness reduction ratios, and reduction percentages of softIGA with respect to IGA in 1D, 2D, and 3D.  
Table \ref{tab:cond2} shows these softIGA reductions with respect to OF-IGA. The condition number reduction ratio $\tilde{\rho}^h$ and percentage $\tilde{\varrho}^h$ of softIGA over OF-IGA are defined similarly as in \eqref{eq:srr} and \eqref{eq:srp}, respectively.
Again, as for Figure~\ref{fig:l1l3err1d}, we use the default values for the softness parameter $\eta$.  In 1D, we use $N=100$ uniform elements.  In 2D, we use a tensor-product mesh with $N=40\times 40$ uniform elements.  In 3D, we use a tensor-product mesh with $N=20\times 20\times 20$ uniform elements.  The first minimal softIGA eigenvalue for these meshes is the same (for the four digits shown in the table) as the ones approximated by IGA or OF-IGA.  For $p=2,3,4$, the stiffness reduction ratios of softIGA with respect to IGA are about 2.1, 2.5, and 3.6, respectively. 
The reduction percentages are about $53\%, 60\%,$ and $72\%$, respectively.  
Table \ref{tab:cond2} shows that the ratios and percentages of softIGA with respect to OF-IGA are smaller. 
Since there are no optical branches (no outliers) in OF-IGA spectra, 
the stiffness and error reduction are not as much as the case in the softFEM setting \cite{deng2021softfem}. 
For a general $C^k, k<p-1,$ $p$-th order IGA element, there are optical branches and one expects that softIGA has larger stiffness and error reductions. 
Due to the tensor-product structure, the ratios and percentages in both Tables \ref{tab:cond} and \ref{tab:cond2} hold in 2D and 3D.
Also, we observe that in 1D, 2D, and 3D, the condition numbers of $p$-th order softIGA are smaller than those of $(p+1)$-th order IGA and the stiffness matrix structures (sparsity) of both scenarios are similar.

Figures~\ref{fig:etap2},~\ref{fig:etap3}, and~\ref{fig:etap4} show the stiffness reduction ratios and the root mean square relative eigenvalue errors with respect to the softness parameter $\eta$, for $p=2,3,4$, respectively.  In all cases, the stiffness reduction ratio increases when $\eta$ increases, which we expect as one removes energy from the bilinear form $\eta s(\cdot, \cdot)$.  However, at some stage, characterized by the root mean square relative eigenvalue error, which is defined as
$$
e_\lambda = \sqrt{ \sum_j \hat e_j^2} = \sqrt{ \sum_j \Big( \frac{\lambda_j^h - \lambda_j }{\lambda_j} \Big)^2},
$$
the overall eigenvalue error increases when $\eta$ increases.  Even though the stiffness of softIGA with the default values of $\eta$ can be further reduced by increasing $\eta$, the eigenvalue errors increase.  Moreover, for large values of $\eta$ (greater than the ones denoted by the vertical lines), the approximated eigenvalues are not well-paired with the physical ones when sorted in ascending order.


\section{Concluding remarks} \label{sec:conclusion}

Our main contributions are as follows. First, we propose softIGA to solve the elliptic differential eigenvalue problem. Second, we derive analytical eigenpairs for the resulting matrix eigenvalue problems, followed by the dispersion error analysis for softIGA. We also establish the coercivity of the softIGA stiffness bilinear form and derive analytical eigenvalue errors. The main advantage of softIGA over the standard IGA is that softIGA leads to stiffness matrices with significantly smaller condition numbers. Hence, the stiffness of the discretized system is significantly reduced. Consequently, this leads to larger stability regions for explicit time-marching schemes. Moreover, on the implementation side, softIGA can take any existing IGA code and only need to add a simple feature to calculate the $p$-th order-derivative jumps for $p$-th order elements of $C^{p-1}$.

A future work direction is applying softIGA to solve explicit dynamics problems. We expect softIGA would improve the simulations in the sense of less computational time by increasing time-step sizes. We observe from Section~\ref{sec:sr} that the stiffness is reduced but the eigenfunction error remains the same. Thus, another exciting and challenging future work is developing softIGA such that stiffness is reduced while increasing the simulation accuracy.
Lastly, the idea of reducing the stiffness of the discretized systems can be applied to other discretization methods and differential operators such as the biharmonic operator. 
For example, softFEM for the second-order operator is extended to the biharmonic operator using both mixed and primal formulation in a straightforward fashion in \cite{deng2021softfem}.
For an operator of the form $\Delta^n, n = 2,3,\cdots$ where the problem is cast into a mixed form of equations where each is of second-order, 
one would expect a direct application of softIGA as the IGA discretization adopted in \cite{deng2019optimal}.
In a primal formulation, the major challenge is to optimize the softness parameter such that the condition numbers are reduced
 while maintaining the accuracy and coercivity of the discrete system.

\section*{Acknowledgments} 

This publication was made possible in part by the Professorial Chair in Computational Geoscience and the Curtin Corrosion Centre at Curtin University. This project has received funding from the European Union's Horizon 2020 research and innovation programme under the Marie Sklodowska-Curie grant agreement No 777778 (MATHROCKS). 



\appendix{}

\section{Matrices for $C^3$-quartic elements} \label{app:p4}

For $C^3$-quartic and $C^4$-quintic OF-IGA $N$ uniform elements in 1D, the mass and stiffness matrices are as follows.
\begin{equation} \label{eq:mk4}
\begin{aligned}
\tilde{K}_4 & = \frac{1}{h}
\begin{bmatrix}
\frac{31}{60} & \frac{19}{120} & -\frac{139}{840} & -\frac{13}{560} & -\frac{1}{5,040}\\[0.2cm]
\frac{19}{120} & \frac{107}{210} & -\frac{17}{560} & -\frac{17}{90} & -\frac{59}{2,520} & -\frac{1}{5,040}  \\[0.2cm]
-\frac{139}{840} &-\frac{17}{560} & \frac{35}{72} & -\frac{11}{360} & -\frac{17}{90} & -\frac{59}{2,520} & -\frac{1}{5,040} \\[0.2cm]
-\frac{13}{560}  & -\frac{17}{90} & -\frac{11}{360}  & \frac{35}{72} & -\frac{11}{360} & -\frac{17}{90} & -\frac{59}{2,520} & -\frac{1}{5,040}  \\[0.2cm]
\ddots & \ddots & \ddots & \ddots & \ddots & \ddots & \ddots & \ddots \\[0.2cm]
\end{bmatrix}_{N\times N}, \\
\tilde{M}_4 &= h
\begin{bmatrix}
 \frac{809}{4,320} & \frac{1,753}{8,640} & \frac{2,351}{60,480} & \frac{167}{120,960} & \frac{1}{362,880} \\[0.2cm]
\frac{1,753}{8,640} & \frac{6,487}{15,120} & \frac{29,411}{120,960} & \frac{913}{22,680} & \frac{251}{181,440} & \frac{1}{362,880} \\[0.2cm]
\frac{2,351}{60,480}  & \frac{29,411}{120,960} & \frac{15,619}{36,288} & \frac{44,117}{181,440} & \frac{913}{22,680} & \frac{251}{181,440} & \frac{1}{362,880} \\[0.2cm]
\frac{167}{120,960}   & \frac{913}{22,680}  & \frac{44,117}{181,440}  & \frac{15,619}{36,288} & \frac{44,117}{181,440} & \frac{913}{22,680} & \frac{251}{181,440} & \frac{1}{362,880} \\[0.2cm]
\ddots  & \ddots & \ddots & \ddots & \ddots & \ddots & \ddots & \ddots  \\[0.2cm]
\end{bmatrix}_{N\times N}, 
\end{aligned}
\end{equation}
%
and in $\mathbb{R}^{{(N-1)\times (N-1)}}$
\begin{equation} \label{eq:mk5}
\begin{aligned}
\tilde{M}_5 &= h
\begin{bmatrix}
 \frac{10,243}{30,240} & \frac{96,823}{403,200} & \frac{50,033}{907,200} & \frac{3,469}{907,200} & \frac{509}{9,979,200} & \frac{1}{39,916,800} \\[0.2cm]
\frac{96,823}{403,200}  & \frac{357,323}{907,200} & \frac{126,469}{518,400} & \frac{1,093}{19,800} & \frac{50,879}{13,305,600} & \frac{509}{9,979,200} & \frac{1}{39,916,800} \\[0.2cm]
\frac{50,033}{907,200}   & \frac{126,469}{518,400}  & \frac{655,177}{1,663,200}  & \frac{1,623,019}{6,652,800} & \frac{1,093}{19,800} & \frac{50,879}{13,305,600} & \frac{509}{9,979,200} & \frac{1}{39,916,800} \\[0.2cm]
\frac{3,469}{907,200} & \frac{1,093}{19,800}  & \frac{1,623,019}{6,652,800}  & \frac{655,177}{1,663,200}  & \frac{1,623,019}{6,652,800} & \frac{1,093}{19,800} & \frac{50,879}{13,305,600} & \frac{509}{9,979,200} & \frac{1}{39,916,800} \\[0.2cm]
\frac{509}{9,979,200} & \frac{50,879}{13,305,600} & \frac{1,093}{19,800}  & \frac{1,623,019}{6,652,800}  & \frac{655,177}{1,663,200}  & \frac{1,623,019}{6,652,800} & \frac{1,093}{19,800} & \frac{50,879}{13,305,600} & \cdots \\[0.2cm]
\ddots  & \ddots & \ddots & \ddots & \ddots & \ddots & \ddots & \ddots  \\[0.2cm]
\end{bmatrix}, \\
%
\tilde{K}_5 & = \frac{1}{h}
\begin{bmatrix}
\frac{8,143}{15,120} & \frac{1,285}{24,192} & -\frac{2,951}{18,144} & -\frac{3,401}{90,720} &  -\frac{25}{18,144} & -\frac{1}{362,880}\\[0.2cm]
\frac{1,285}{24,192} & \frac{34,103}{90,720} & \frac{5,671}{362,880} & -\frac{31}{189} & -\frac{907}{24,192} & -\frac{25}{18,144} & -\frac{1}{362,880}  \\[0.2cm]
-\frac{2,951}{18,144} &\frac{5,671}{362,880} & \frac{809}{2,160} & \frac{1}{64} & -\frac{31}{189} & -\frac{907}{24,192} & -\frac{25}{18,144} & -\frac{1}{362,880}  \\[0.2cm]
-\frac{3,401}{90,720}  & -\frac{31}{189} & \frac{1}{64} & \frac{809}{2,160} & \frac{1}{64} & -\frac{31}{189} & -\frac{907}{24,192} & -\frac{25}{18,144} & -\frac{1}{362,880}  \\[0.2cm]
 -\frac{25}{18,144} & -\frac{907}{24,192}  & -\frac{31}{189} & \frac{1}{64} & \frac{809}{2,160} & \frac{1}{64} & -\frac{31}{189} & -\frac{907}{24,192} &  \cdots  \\[0.2cm]
\ddots & \ddots & \ddots & \ddots & \ddots & \ddots & \ddots & \ddots \\[0.2cm]
\end{bmatrix},
%
\end{aligned}
\end{equation}
where all the missing entries are either zeros or completed using the property of matrix symmetry and persymmetry. 
The softness bilinear form $s(\cdot, \cdot)$ leads to the matrix
\begin{equation} \label{eq:s4}
\begin{aligned}
S_4 & = 
\begin{bmatrix}
462 & -330 & 165 & -55 & 11 & -1 \\
-330 & 297 & -220 & 121 & -45 & 10 & -1 \\
165 & -220 & 253 & -210 & 120 & -45 & 10 & -1 \\
-55 & 121 & -210 & 252 & -210 & 120 & -45 & 10 & -1 \\
11 & -45 & 120 & -210 & 252 & -210 & 120 & -45 & 10 & -1 \\
 & \ddots & \ddots & \ddots & \ddots & \ddots & \ddots & \ddots & \ddots & \ddots\\
\end{bmatrix}_{N\times N}, \\
S_5 & = 
\begin{bmatrix}
429 & -572 & 429 & -208 & 65 & -12 & 1 \\
-572 & 858 & -780 & 494 & -220 & 66 & -12 & 1 \\
429 & -780 &  923 & -792 & 495 & -220 & 66 & -12 & 1 \\
-208 & 494 & -792 & 924 & -792 & 495 & -220 & 66 & -12 & 1 \\
 &  \ddots & \ddots & \ddots & \ddots & \ddots & \ddots & \ddots & \ddots & \ddots \\
\end{bmatrix}_{(N-1)\times (N-1)},
\end{aligned}
\end{equation}
where the entries near the right boundary are such that the matrix is symmetric and persymmetric.

\section{Several transformation matrices $T_p$} \label{app:tp}

For softIGA with $p=2,3,4,5$ on N uniform elements in $\Omega=[0,1]$, the transformation matrices are as follows. 

\begin{equation} \label{eq:t2t3}
T_2 = 
\begin{bmatrix}
 \frac{5}{8} & \frac{1}{8} \\[0.2cm]
\frac{1}{8} & \frac{3}{4} & \frac{1}{8} \\[0.2cm]
& \frac{1}{8} & \frac{3}{4} & \frac{1}{8} \\[0.2cm]
 &  & \ddots & \ddots & \ddots &  \\[0.2cm]
  & &  & \frac{1}{8} & \frac{3}{4} & \frac{1}{8}  \\[0.2cm]
&&& & \frac{1}{8} & \frac{5}{8}  \\
\end{bmatrix}_{N\times N}, \qquad
T_3 =
\begin{bmatrix}
 \frac{2}{3} & \frac{1}{6} \\[0.2cm]
\frac{1}{6} & \frac{2}{3} & \frac{1}{6} \\[0.2cm]
& \frac{1}{6} & \frac{2}{3} & \frac{1}{6} \\[0.2cm]
 &  & \ddots & \ddots & \ddots &  \\[0.2cm]
  & &  & \frac{1}{6} & \frac{2}{3} & \frac{1}{6}  \\[0.2cm]
&&& & \frac{1}{6} & \frac{2}{3}  \\
\end{bmatrix}_{(N-1)\times (N-1)}.
\end{equation} 

\begin{equation} \label{eq:t4t5}
T_4 =
\begin{bmatrix}
 \frac{77}{192} & \frac{25}{128} & \frac{1}{384} \\[0.2cm]
\frac{25}{128} & \frac{115}{192} & \frac{19}{96}  & \frac{1}{384} \\[0.2cm]
\frac{1}{384} & \frac{19}{96} & \frac{115}{192} & \frac{19}{96}  & \frac{1}{384} \\[0.2cm]
 &  \ddots & \ddots & \ddots & \ddots & \ddots  \\[0.2cm]
  & &  \frac{1}{384} & \frac{19}{96} & \frac{115}{192} & \frac{25}{128} \\[0.2cm]
&& &  \frac{1}{384} & \frac{25}{128} & \frac{77}{192}  \\
\end{bmatrix}_{N\times N}, 
T_5  =
\begin{bmatrix}
 \frac{13}{24} & \frac{13}{60} & \frac{1}{120} \\[0.2cm]
\frac{13}{60} & \frac{11}{20} & \frac{13}{60}  & \frac{1}{120} \\[0.2cm]
\frac{1}{120} & \frac{13}{60} & \frac{11}{20} & \frac{13}{60}  & \frac{1}{120} \\[0.2cm]
 &  \ddots & \ddots & \ddots & \ddots & \ddots  \\[0.2cm]
&& \frac{1}{120} & \frac{13}{60} & \frac{11}{20} & \frac{13}{60}  \\[0.2cm]
&&& \frac{1}{120} & \frac{13}{60} & \frac{13}{24}  \\[0.2cm]
\end{bmatrix}_{(N-1)\times (N-1)}.
\end{equation}

\end{document}